\documentclass[11pt,A4paper]{article}

\usepackage[left=30mm,right=30mm,top=25mm,bottom=30mm]{geometry}
\usepackage{labelfig}
\usepackage{float}
\usepackage{xfrac}

\usepackage[percent]{overpic}

\usepackage{color}
\usepackage{amsthm,amsmath,amssymb}
\usepackage{booktabs}
\usepackage{mathpazo}
\usepackage{microtype}
\usepackage{overpic}
\usepackage{bm}
\usepackage{sectsty}
\usepackage[pdftitle={PDFTitle},pdfauthor={Hugo Parlier},ocgcolorlinks,linkcolor=linkred,citecolor=linkred,urlcolor=linkblue]{hyperref}

\usepackage{ifpdf}

\usepackage{mathtools}

\definecolor{linkred}{RGB}{128,0,128}
\definecolor{linkblue}{RGB}{16, 78, 139}

\usepackage[hang,flushmargin]{footmisc}
\usepackage{enumitem}
\usepackage{titlesec}
	\titlespacing{\section}{0pt}{12pt}{0pt}
	\titlespacing{\subsection}{0pt}{6pt}{0pt}

\titlelabel{\thetitle.\quad}

\makeatletter 

\long\def\@footnotetext#1{% 
\H@@footnotetext{% 
\ifHy@nesting 
\hyper@@anchor{\@currentHref}{#1}% 
\else 
\Hy@raisedlink{\hyper@@anchor{\@currentHref}{\relax}}#1% 
\fi 
}}

\def\@footnotemark{% 
\leavevmode 
\ifhmode\edef\@x@sf{\the\spacefactor}\nobreak\fi 
\H@refstepcounter{Hfootnote}% 
\hyper@makecurrent{Hfootnote}% 
\hyper@linkstart{link}{\@currentHref}% 
\@makefnmark 
\hyper@linkend 
\ifhmode\spacefactor\@x@sf\fi 
\relax 
}% 

\ifFN@multiplefootnote% 
\renewcommand*\@footnotemark{% 
\leavevmode 
\ifhmode 
\edef\@x@sf{\the\spacefactor}% 
\FN@mf@check 
\nobreak 
\fi 
\H@refstepcounter{Hfootnote}% 
\hyper@makecurrent{Hfootnote}% 
\hyper@linkstart{link}{\@currentHref}% 
\@makefnmark 
\hyper@linkend 
\ifFN@pp@towrite 
\FN@pp@writetemp 
\FN@pp@towritefalse 
\fi 
\FN@mf@prepare 
\ifhmode\spacefactor\@x@sf\fi 
\relax% 
}% 
\fi

\makeatother 

\theoremstyle{plain}
\newtheorem{theorem}{Theorem}[section]
\newtheorem{proposition}[theorem]{Proposition}
\newtheorem{lemma}[theorem]{Lemma}
\newtheorem{corollary}[theorem]{Corollary}

\makeatletter
\newtheorem*{rep@theorem}{\rep@title}
\newcommand{\newreptheorem}[2]{%
\newenvironment{rep#1}[1]{%
 \def\rep@title{#2 \ref{##1}}%
 \begin{rep@theorem}}%
 {\end{rep@theorem}}}
\makeatother

\newreptheorem{theorem}{Theorem}
\newreptheorem{corollary}{Corollary}

\theoremstyle{definition}
\newtheorem{definition}[theorem]{Definition}

\newtheorem{remark}[theorem]{Remark}
\newtheorem{observation}[theorem]{Observation}

\newcommand{\R}{{\mathbb R}}

\newcommand{\Hyp}{{\mathbb H}}
\newcommand{\N}{{\mathbb N}}

\newcommand{\C}{{\mathbb C}}

\newcommand{\CC}{{\mathcal C}}

\newcommand{\OO}{{\mathcal O}}

\newcommand{\M}{{\mathcal M}}
\newcommand{\PSL}{{\rm PSL}}

\newcommand{\arcsinh}{{\,\rm arcsinh}}
\newcommand{\arccosh}{{\,\rm arccosh}}

\newcommand{\measure}{\lambda}
\newcommand{\bigmeasure}{\lambda_{\vec{k}}}

\newcommand{\PP}{\mathcal P}

\sectionfont{\large \bfseries}
\subsectionfont{\normalsize}

\long\def\symbolfootnote[#1]#2{\begingroup%
\def\thefootnote{\fnsymbol{footnote}}\footnote[#1]{#2}\endgroup}

\def\blfootnote{\xdef\@thefnmark{}\@footnotetext}

\begin{document}

{\Large \bfseries Prime orthogeodesics, concave cores and families of identities on hyperbolic surfaces}

{\large Ara Basmajian\symbolfootnote[1]{\normalsize Supported by a PSC-CUNY Grant and a grant from the Simons foundation (359956, A.B.)}, Hugo Parlier\symbolfootnote[2]{\normalsize Supported by the Luxembourg National Research Fund OPEN grant O19/13865598.}
 and Ser Peow Tan\symbolfootnote[3]{\normalsize Supported by the National Univ. of Singapore academic research grant R-146-000-289-114.
}}

{\bf Abstract.} 
We prove and explore a family of identities relating lengths of curves and orthogeodesics of hyperbolic surfaces. These identities hold over a large space of metrics including ones with hyperbolic cone points, and in particular, show how to extend a result of the first author to surfaces with cusps. One of the main ingredients in the approach is a partition of the set of orthogeodesics into sets depending on their dynamical behavior, which can be understood geometrically by relating them to geodesics on orbifold surfaces. These orbifold surfaces turn out to be exactly on the boundary of the space in which the underlying identity holds.
\vspace{2cm}

%%%%%%%%%INTRODUCTION%%%%%%%%

\section{Introduction}\label{sec:intro}

In the study of moduli spaces of hyperbolic surfaces, a number of identities relating the lengths of geodesics have been explored over the past few decades. These sums have the remarkable property of remaining constant over an entire moduli space while the individual summands vary continuously. 

A prime example is the McShane identity \cite{McShane} for cusped surfaces, generalized by Mirzakhani to surfaces with geodesic boundary \cite{Mirzakhani} and generalized to surfaces with boundary and/or cone points by Tan, Wong and Zhang \cite{Tan-Wong-Zhang}. The original McShane identity was for once-punctured tori but was generalized to once-punctured genus $g$ surfaces and states that
$$
\sum_{P} \frac{2}{e^{\frac{\ell(\alpha)+\ell(\beta)}{2}} +1} = 1,
$$
where the sum is taken over all embedded (geodesic) pairs of pants $P$ with boundary elements the puncture and simple geodesics $\alpha$ and $\beta$. 

We highlight three ingredients of the identity. First of all the index set: in this case the set of embedded pairs of pants. To each element of the index set is associated a function which depends on the geometry of this element: here to the pair of pants one associates a value which depends on its cuff lengths. Finally, the sum of these functions is equal to the measure of a geometric feature of the surface which remains stable over the moduli space (here the $1$ is the length of a horocycle). This philosophy of breaking down the measure of a geometric feature into elements of full measure with complementary measure $0$ is found in all of the identities which will be discussed. In particular, in the generalizations by Mirzakhani and Tan-Wong-Zhang, all of these features are present, and the only thing that changes slightly is that the full measure is now the measure of the boundary, either length or angle. Appropriately scaled, putting these identities all together gives a continuous set of identities when one varies the boundary element from being a simple closed geodesic to a cusp and then to a cone-point. 

Around the same time, Basmajian showed a seemingly similar identity, this time for surfaces with boundary geodesics. In the case of a surface $X$ with a single boundary curve $\beta$, the identity is as follows:
$$
\sum_{\mu \in \OO(X)} 2 \log\left(\coth\frac{\ell(\mu)}{2}\right) = \ell(\beta)
$$
where here the sum is taken over all {\it orthogeodesics} of $X$. These are the immersed oriented geodesic segments twice perpendicular to the boundary. In particular, note that they are not necessarily simple, and in fact most are not. Also note that, in contrast to the previous identities, the summands only depend on the length of a single geodesic while the right-hand side of the equation does not. In particular, this means that knowing the set of lengths of orthogeodesics (with multiplicities) provides the boundary length, however impractical it might be to compute boundary length with this method. The identity fails to hold for surfaces without any boundary geodesics, and fails even for cusps using the same scaling trick to obtain the McShane identity via the Mirzakhani or Tan-Wong-Zhang identities. 

The same index set appears in an identity by Bridgeman, again for surfaces with at least one geodesic boundary curve. Bridgeman's identity results from a decomposition of the unit tangent bundle of a surface into subsets that correspond to different orthogeodesics. In turn the idea of decomposing the unit tangent bundle was used by Luo and Tan \cite{Luo-Tan} to obtain an identity which also works for closed hyperbolic surfaces, and where the index set is (necessarily) very different. In both cases, the measure comes from the standard volume on the unit tangent bundle, but the Luo-Tan identity decomposes the unit tangent bundle into sets indexed by embedded geodesic pants and one-holed tori.

We present here a family of identities which generalize the orthospectrum identity by Basmajian into an identity that extends beyond the case of surfaces with geodesic boundary. Our identities are also decompositions of the boundary of a surface, but in order to be able to look at surfaces with cusps or cone points, we decompose the boundary of a subset of the surface obtained by cutting off collar regions of boundary elements. For this we introduce a notion of {\it natural} collars which depends on the behavior of geodesics in the neighborhood of a boundary element of a surface $X$ and is indexed by a sequence of integers $\vec{k} = (k_1,\hdots,k_n)$, one for each boundary element. We refer to this set of integers as a {\it grading} and an individual $k_i$ as a {\it grade}, and a grade can be {\it infinite}. We call the complementary regions to these natural collars the ($\vec{k}$-relative) concave core (denoted $V_{\vec{k}}(X)$) and our identities are a decomposition of the boundary of the concave cores. 

A main feature is the index sets. Note that the index set for the McShane identity is usually thought of as a set of embedded pants, but can also be viewed as the set of {\it simple} orthogeodesics, a subset of $\OO(X)$. The measures associated to each simple orthogeodesic depend on more than the length of the orthogeodesic however. One interpretation of the McShane identity is to view it as a way of grouping elements of $\OO(X)$ together in terms of the initial geodesic behavior of orthogeodesics. We also consider ways of grouping elements of $\OO(X)$, but in a variety of ways which provide an exhaustive family of subsets of $\OO(X)$ and which we describe from topological and geometric viewpoints. 

Given $\vec{k}$, orthogeodesics of $V_{\vec{k}}(X)$ (denoted $\OO_{\vec{k}}(X)$) correspond to the subset of $\OO(X)$ which never wrap $k_i$ times around boundary element $i$, so as sets, don't depend on the geometry of $X$. From a topological viewpoint, we associate to each such orthogeodesic $\mu$, which lies between two boundary curves, say $\alpha$ and $\beta$, an immersed pair of pants. The boundary curves of the immersed pair of pants are given by concatenations of $\mu$, $\alpha$ and $\beta$, depending on the gradings of $\alpha$ and $\beta$. For instance if both $\alpha$ and $\beta$ have grade $1$, the boundary curves of the immersed pair of pants are $\alpha$, $\beta$ and $\alpha *\mu*\beta*\mu^{-1}$. In general, we shall consider sets of immersed pair of pants $\PP_{\vec{k}}(X)$ which are associated to a grading $\vec{k}$ (see Section \ref{sec:setup} for details). Unlike embedded pairs of pants, there are infinitely many immersed pair of pants inside any given immersed pair of pants. However, certain elements of $\PP_{\vec{k}}(X)$ are special: they are not contained inside any other element, and we call them, and their associated orthogeodesic, $\vec{k}$-prime. We can now state our abstract identity which works for surfaces satisfying a condition we shall outline below.

\begin{theorem}\label{thm:abstract}
Any $X \in \M^{\vec{k}}(\Sigma)$ satisfies
$$
\sum_{\mu \in \OO'_{\vec{k}}(X)} \bigmeasure(\mu) = \ell\left(\partial V_{\vec{k}}(X)\right)
$$
\end{theorem}

The set $\OO'_{\vec{k}}(X)$ is the set of $\vec{k}$-prime orthogeodesics and the measure $\lambda_{\vec{k}}$ associated to the orthogeodesic $\mu$ depends explicitly on the geometry of the associated immersed pair of pants. This will be detailed, and we will give explicit values later in the introduction, but first we concentrate on the set of hyperbolic metrics $\M^{\vec{k}}(\Sigma)$. $\Sigma$ is a topological surface of finite type, thus of genus $g$ with $n$ numbered marked points. The set $\M^{\vec{k}}(\Sigma)$ is a moduli space of hyperbolic metrics that we call admissible: this is the requirement that, for $i=1,\hdots,n$, the $i$-th boundary element be realized as either a cone point of angle $\leq \frac{\pi}{k_i}$, or a cusp or a simple closed geodesic. Note that if a grade is infinite, as usual, a cone point of angle $0$ is the same thing as a geodesic of length $0$, which is simply a cusp. Finally, observe that a surface $X$ is in $\M^{\vec{k}}(\Sigma)$ if and only if $V_{\vec{k}}(X)$ is defined.

Certain metrics in $\M^{\vec{k}}(\Sigma)$ are special: these are the limit metrics where all boundaries are cone points of angle $\frac{\pi}{k_i}$ for all $i=1,\hdots,n$. The resulting metrics are orbifold surfaces, and we call them {\it model surfaces}, denote them $M_{\vec{k}}$, and they play an important role in our approach. 
Note that all boundary elements of $V_{\vec{k}}(M_{\vec{k}})$ are exactly of length $0$, so the identity for any model surface is void of content, or said otherwise, our identities hold and have content for all surfaces up until these special limit cases. They also play another important role, as their geodesics help us understand the notion of $\vec{k}$-primality. 

\begin{theorem}
An orthogeodesic $\mu \in \OO(X)$ is in $\OO'_{\vec{k}}(X)$ if and only if it corresponds to a properly immersed geodesic path on any model surface $M_{\vec{k}}$.
\end{theorem}

Properly immersed just means that the unique geodesic minimizer of the path has only its endpoints in a cone point. The above theorem is really our trick to understand the index set in a geometric fashion, as opposed to trying to argue topologically whether certain immersed pairs of pants satisfy inclusion properties. 

To prove the above identity requires understanding the index set, but also requires understanding the measures $\bigmeasure(\mu)$, which are intervals of the boundary of $V_{\vec{k}}(X)$, and in particular why they are disjoint and why their complement is measure $0$.The measure $0$ part uses some kind of ergodicity property, which becomes trickier when the surface has cone points. We provide a self-contained proof which doesn't require using a previously known ergodicity type theorem. Although one could probably appeal to known results, our approach is tailored for the problem at hand so as to provide additional insight into why the identity really works (see Section \ref{sec:ergodicity}).

The next part is computing the measures explicitly. The intervals we use are similar to certain intervals used in McShane type identities, only here we associate them to {\it any} orthogeodesic. As we have a dual interpretation (in terms of orthogeodesics and immersed pairs of pants) this leads to different expressions. It turns out that one nice way of expressing the identity is in terms of half-traces (of the corresponding elements in $\PSL_2(\R)$). Thus the half-trace of a cone point of angle $\theta$ is $\cos(\theta/2)$ and the half-trace of a geodesic of length $\ell$ is $\cosh(\ell/2)$. 

Here is one of the expressions of the measures in terms of half-traces. 
\begin{theorem}
The measure associated to an orthogeodesic $\mu$ that leaves from a geodesic $\alpha$ and goes to a boundary element $\beta$ is 
$$
\bigmeasure(\mu)= \frac{2 T_{k_\alpha}(a)}{\sqrt{T_{k_\alpha}(a)^2-1}}
\arccosh\left(\frac{T_{k_\alpha}(a)^2(c-\sqrt{c^2-1})+T_{k_\alpha}(a)T_{k_\beta}(b) + \sqrt{c^2-1}}{\sqrt{p(T_{k_\alpha}(a),T_{k_\beta}(b),c)}}\right).
$$
where $c$ is the half-trace of the boundary of the immersed pair of pants associated to $\mu$, $a$ is the half-trace of $\alpha$, $b$ is the half-trace of $\beta$, $T_i$ is the Chebyschev polynomial of the first kind of degree $i$ and $p$ is the polynomial 
$$
p(x,y,z) = x^2 + y^2 +z^2 - 2xyz +1.
$$
\end{theorem}
There are similar expressions in all of the different geometric situations (see Section \ref{sec:measures}). They are, in some sense, all the same and in particular provide a continuous family of measure associated to an orthogeodesic leaving from $\alpha$ as you transform $\alpha$ from a geodesic, to a cusp and then to a cone point.

When the surface has geodesic boundary, as the grades grow, the measures converge to the measure in the orthospectrum identity of Basmajian. In fact we show:
 
\begin{theorem}
Let $X$ be a hyperbolic surface with geodesic boundary. 
The identity associated to the grade $\vec{k}$ converges to the orthospectrum identity, as 
$\vec{k} \rightarrow (\infty,...,\infty)$. 
\end{theorem}

The identity takes one of its simplest forms when the boundary elements are all cusps. Here $\gamma_\mu$ is the geometric length of the boundary curve of the immersed pair of pants associated to $\mu$, and this is the "simplest" form of the identity, because all of the grades are set to be $1$ (we use the notation $\vec{1} = (1,\hdots,1)$). 

\begin{theorem}
Let $X$ be a surface of genus $g$ and $n$ cusps with $\chi(X) = 2-2g -n \leq -1$ and $(g,n)\neq (0,3)$. Then 
$$
\sum_{\mu \in \OO'_{\vec{1}}(X)} \frac{2}{e^{\frac{\gamma_\mu}{2}}+1} = n.
$$
\end{theorem}

A particular case of the above identity is when $(g,n)=(0,4)$, and in this case the above identity is in fact the McShane identity for the four holed sphere. This is one explanation for the striking similarity between the above measures and those that appear in the original McShane identity. However, for all other topological types, including for punctured tori, the identities are very different. Note that although the above identity does not work for pairs of pants, this is due to the low grading: one feature of our identities is that we do get identities for pairs of pants, provided the grading satisfies certain lower bounds. As an example of these, we get the following identity for thrice punctured spheres. 

\begin{theorem}
Let $Y$ be the thrice punctured sphere and let $k\leq l \leq m$ be positive integers so that $ k + l + m >4$. Then 
$$
\sum_{\mu \in \OO'_{(k,l,m)}(Y)} \frac{2}{e^{\frac{\gamma_\mu}{2}}+1} = \frac{1}{k}+\frac{1}{l}+\frac{1}{m}.
$$
\end{theorem}
An advantage of dealing with orthogeodesics, is that they behave "well" under covers, in that, unlike curves, they lift to other orthogeodesics of the same length. By taking a cover of a surface, you get a new identity. In fact, in the above theorem for the thrice punctured sphere, the case where $k=1, l=2$ and $m=2$ is really a particular case of the previous theorem for $(g,n)=(0,4)$, hence another instance of the McShane identity, but for thrice punctured spheres. 

{\bf Organization.}

This article is organized as follows. 

We give some general definitions and a table of the notation we use in Section \ref{sec:basicsetup}. In Section \ref{sec:setup} we introduce the main objects we study, namely orthogeodesics, immersed pants and natural collars of boundary elements. Section \ref{sec:prime} is about the primality of orthogeodesics and their relationship to the geometry of model surfaces. In Section \ref{sec:ergodicity}, we prove the abstract identity, which includes understanding the dynamics of orthogeodesics, but also the topology of immersed pants. In Section \ref{sec:measures}, we compute the measures associated to orthogeodesics, which requires computing them in different geometric situations, and in terms of different types of input. We are then able to state quantified versions of the identities in Section \ref{sec:identities}, and we end the article with some additional observations about the identities including growth aspects of the index sets in Section \ref{sec:finalobservations}.

{\bf Acknowledgments.}

We are very grateful for the support of the Bernoulli Center of the EPFL in Lausanne for hosting us in January 2019. The authors acknowledge support from U.S. National Science Foundation grants DMS 1107452, 1107263, 1107367 ”RNMS: GEometric structures And Representation varieties” (the GEAR Network).

%%%%%%%%%%%NOTATION%%%%%%%%%%%%%%%%
\section{Basic setup and notation} \label{sec:basicsetup}

\subsection{Basic setup} \label{subsec: Basic setup}

Consider a (finite type) oriented topological surface $\Sigma = \Sigma_{g,n}$ of genus $g$ with $n>0$ boundary simple closed curves. Although from a topological point of view, this is the same as considering punctures or marked points, we want to think of the boundary elements as oriented loops. We require that the orientation coincides with a given orientation of the surface: for instance the surface is always to the right of a boundary loop. More generally, unless explicitly stated arcs and curves are considered as oriented objects. 

We are interested in geometries on $\Sigma$, that for the most part will be hyperbolic metrics with boundary elements. The geometric realization of a boundary element may be either a simple closed geodesic, a cusp or a cone-point, thus giving rise to finite area, but possibly geodesically incomplete metrics (sometimes referred to as the convex core of the surface). Along a geodesic boundary, it will sometimes be convenient to think of the metric as extending beyond the boundary by adding infinite funnels in the obvious way.

Boundary elements are numbered from $1$ to $n$, and as mentioned in the introduction, to each boundary element we associate a {\it grade}, which is an element of $\N \cup \{\infty\} = \{1,2,\hdots\}\cup \{\infty\}$. The collection of grades we call a {\it grading} and denote it by $\vec{k}= \{k_1,\hdots,k_n\}$. 

The hyperbolic metrics we will be considering are the following.

\begin{definition} Fix a grading $\vec{k}$. A $\vec{k}$-admissible metric on $\Sigma$ is a hyperbolic metric where boundary elements are either simple closed geodesics, cusps, or cone points of angle $\theta \leq \frac{\pi}{k_i}$ for the boundary element $i$. We denote the space of $\vec{k}$-admissible metrics on $\Sigma$ by ${\mathcal M}^{\vec{k}} (\Sigma)$.\end{definition}

Note that when $k_i=\infty$, an admissible metric may have a cusp or a simple closed geodesic as its $i$-th boundary component. 

\subsection{Notation}

We collect the various notation we use and the first place they appear in Table \ref{Table:notation}. For ease of notation and clarity, we often denote a curve and its length with the same symbol. The context should always be clear. We usually use the symbol $\mu$ for orthogeodesics joining two not necessarily distinct boundary elements
$\alpha$ (initial) and $\beta$ (terminal). 

%%%%%%%%%%%%Table: Notation%%%%%%%%%%%%%%%%%%%%%
 \begin{table}[H]
\caption{Definitions and notation}
\begin{center}
 \begin{tabular}{| l | c | c |}
\hline\hline 
 {\bf Definition} & {\bf Section} &{\bf Notation} \\
 \hline
\hline

Grading &\ref{sec:intro} &$\vec{k} = (k_1,\hdots,k_n)$\\
\hline
$\vec{k}$-admissible surfaces & \ref{subsec: Basic setup} &${\mathcal M}^{\vec{k}}(\Sigma)$\\
\hline
Immersed pairs of pants& \ref{subsec: Orthogeodesics and immersed pairs of pants} &$\PP(X)$\\
\hline
Boundary curves of immersed pants&\ref{subsec: Orthogeodesics and immersed pairs of pants} &$\partial\PP(X)$\\
\hline
Oriented orthogeodesics on $X$
&\ref{subsec: Orthogeodesics and immersed pairs of pants} &$\OO(X)$ \\
\hline
Unoriented orthogeodesics on $X$ &\ref{subsec: Orthogeodesics and immersed pairs of pants} & $\overline{\OO}(X)$\\
\hline
 $k$-th natural collar of $\delta$ &
 \ref{subsec: natural collars and concave cores} & $\CC_k(\delta)$ \\
\hline
 Concave core&\ref{subsec: natural collars and concave cores} &$V_{\vec{k}}(X)$ \\
\hline

Orthogeodesics on the concave core&
\ref{subsec: subsets of orthogeodesics} &$\OO_{\vec{k}}(X)$ \\
\hline
$\vec{k}$-prime orthogeodesics&\ref{subsec:Primality in full generality}
 &$\OO'_{\vec{k}}(X)$\\
\hline
$ \vec{k}$-prime pairs of pants&\ref{subsec:Primality in full generality}
 & $\PP'_{\vec{k}}(X)$\\
\hline

$\vec{k}$-model surface&\ref{subsec: Model surfaces} & $M_{\vec{k}}$ \\
\hline

Measure on $\partial{X}$&\ref{subsec:Computations} &$\lambda$ \\
\hline
Measure on $\partial{V_{\vec{k}}}$ & \ref{subsec:Computations} & $\lambda_{\vec{k}}$ \\
\hline

\end{tabular} 
\end{center}
\label{Table:notation}
\end{table}

%%%Orthogeodesics, natural collars and concave cores%%%

\section{Orthogeodesics, natural collars and concave cores}\label{sec:setup}

\subsection{Orthogeodesics and immersed pairs of pants}
\label{subsec: Orthogeodesics and immersed pairs of pants}
The index sets are completely crucial in our investigation and for this reason we present several ways of defining them and show why the definitions are equivalent.

An orthogeodesic of a hyperbolic surface $X$ is an immersed geodesic segment orthogonal in both endpoints to the boundary $\partial X$. We will be using both oriented and non-oriented orthogeodesics. Note that if one or both of the boundary elements it joins is a cusp, then it is of infinite length. If one of its ends is a cone-point, resp. a cusp, saying that it is orthogonal to the boundary might seem a bit odd, but observe that it is orthogonal to circles around the cone point, resp. horocycles around a cusp. For $X$ a hyperbolic surface, we denote by $\OO(X)$ the set of oriented orthogeodesics. Although we will use them less, we set $\overline{\OO}(X)$ to be the set of unoriented orthogeodesics. 

To each orthogeodesic $\mu \in \OO(X)$, there is a particular associated closed geodesic $\gamma_\mu$ defined as follows. The orthogeodesic $\mu$ goes between two boundary elements $\alpha$ and $\beta$ which we think of as oriented simple closed curves. (Note that as we are defining a homotopy class, this can all be done on $\Sigma$ as well.) The homotopy class $\alpha * \mu * \beta * \mu^{-1}$ of a closed curve corresponds on $X$ to a unique oriented and primitive closed geodesic which we denote $\gamma_\mu$. Observe that $\gamma_\mu$ is the boundary of an immersed geodesic pair of pants $P_\mu$ where the other two boundary elements are $\alpha$ and $\beta$. An example is illustrated in Figure \ref{fig:immersedpants}.

\begin{figure}[h]
%\ShowGrid
\leavevmode \SetLabels
%\L(.37*.92) $\alpha$\\%
%\L(.62*.92) $\beta$\\
%\L(.49*.11) $\gamma$\\
%\L(.49*.6) $c$\\
\endSetLabels
\begin{center}
\AffixLabels{\centerline{\includegraphics[width=5cm]{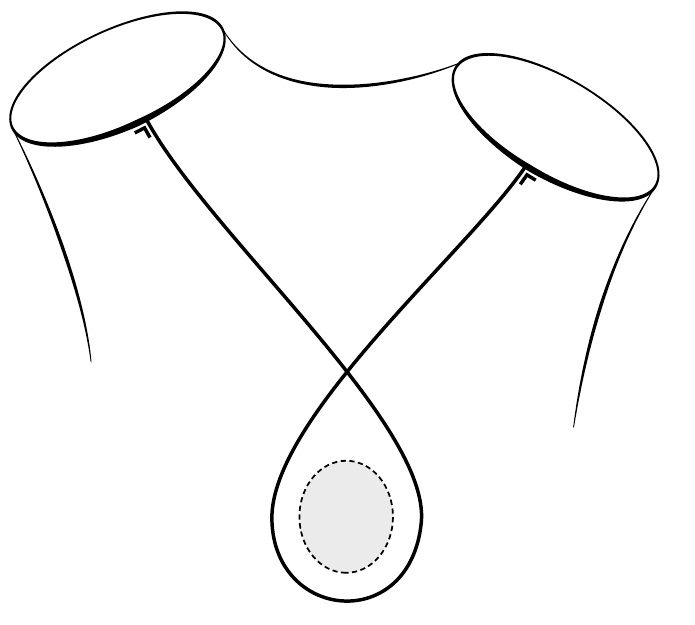}\hspace{50pt}\includegraphics[width=5cm]{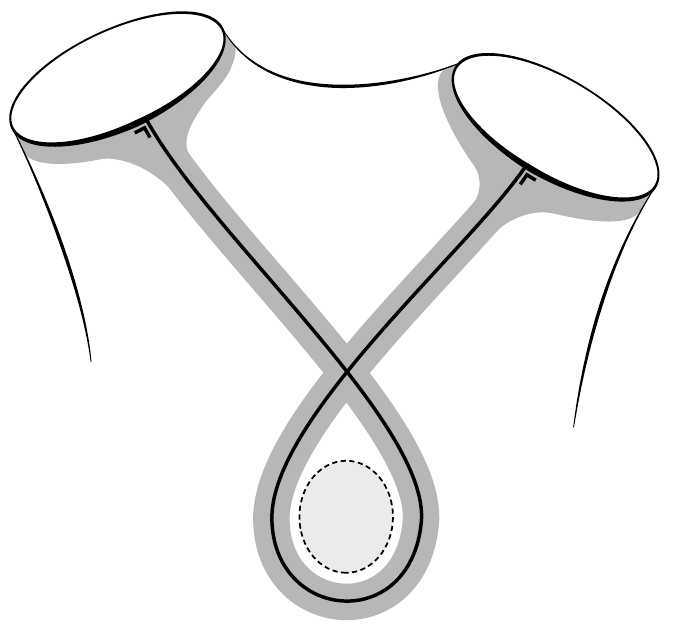}}}
\vspace{-24pt}
\end{center}
\caption{An orthogeodesic and its associated immersed pants in grey} \label{fig:immersedpants}
\end{figure}

$\PP(X)$ will denote the set of immersed pants $P_\mu$ and $\partial \PP(X)$ will denote 
the set of closed curves $\gamma_\mu$ as defined above.

Also observe that both an oriented orthogeodesic $\mu$ and the orthogeodesic with the opposite orientation $\mu^{-1}$ define the same curves: that is
$$
\gamma_{\mu}= \gamma_{\mu^{-1}}.
$$
Hence there is a one-to-one correspondence between {\it immersed pants} $\PP(X)$ (or their boundary curves $\partial\PP(X)$) and {\it unoriented orthogeodesics} $\overline{\OO}(X)$ but a one-to-two correspondence
$$\gamma_\mu \in \partial\PP(X) \leftrightarrow \mu\in \OO(X)$$
between boundary curves of {\it immersed pants} and {\it oriented orthogeodesics}. 

These correspondences will be crucial in the sequel.

\subsection{Natural collars and concave cores}
\label{subsec: natural collars and concave cores}

The geometry of a surface around boundary simple closed geodesics, cone points and cusps is quite simple as they locally look like geodesic cylinders in the collar regions around them. There are a number of ways to choose collar regions. Here we define a family of such regions in terms of the behavior of geodesics. 

The first element of our family is the ``usual" collar which appears naturally in many contexts.

\begin{definition} The {\it first natural collar} of a boundary component is
\begin{itemize} \item the open horoball with boundary length
$2$ around a cusp boundary element,
\item the set of points at distance strictly less than $\arcsinh\left( \frac{1}{\sin\left(\frac{\theta}{2}\right)}\right)$ from a boundary cone point of angle $\theta$,
\item the set of points at distance strictly less than $\arcsinh\left( \frac{1}{\sinh\left(\frac{\ell}{2}\right)}\right)$ if the boundary element is a simple geodesic of length $\ell$.
\end{itemize}\end{definition}

Note that above definition requires that $\theta \leq \pi$ and if $\theta=\pi$, then the natural collar is empty. The regions above are always embedded in the surface and for different boundary elements, are pairwise disjoint. The reason we use the adjective "natural" is due to the following proposition, which is well known to experts.

\begin{proposition}\label{prop:natural}
A complete geodesic path enters the first natural collar if and only if it forms a loop around the boundary component.
\end{proposition}

In particular, a geodesic loop which turns at least one time around the boundary component enters the first natural collar. We prove a more general proposition in the next section. This proposition suggests a natural generalization of the notion of natural collar (hence the use of the word "first" before). A geodesic that enters the first natural collar either hits the boundary element or wraps around it at least once. For a positive integer $k$, we want to define the $k$-th natural collar as the geometric region in which geodesics either hit the boundary element or wrap around it at least $k$ times. 

\begin{definition} The {\it $k$-th natural collar} $\CC_k(\delta)$ of a boundary component $\delta$ is
\begin{itemize} \item an open horoball of boundary length
$\frac{2}{k}$ around if $\delta$ is a cusp,
\item the set of points at distance strictly less than $\arccosh\left( \frac{1}{\sin\left(\frac{k \theta}{2}\right)}\right)$ if $\delta$ is a cone point of angle $\theta$,
\item the set of points at distance strictly less than $ \arcsinh\left( \frac{1}{\sinh\left(\frac{k \ell}{2}\right)}\right)$ if $\delta$ is a simple geodesic of length $\ell$.
\end{itemize}\end{definition}

\begin{figure}[htbp]
\begin{center}
\includegraphics[width=9cm]{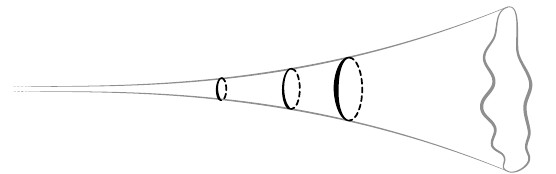}
\caption{The natural collars of a cusp}
\label{fig:naturalcollars}
\end{center}
\end{figure}

Note that the $k$-th natural collar makes sense for all $k$ for boundary geodesics or cusps but only makes sense for cone points of angle less than or equal to $\frac{\pi}{k}$, and it is empty in case of equality. Similarly the $\infty$-th natural collar is empty for both cusps or boundary geodesics. However in the latter case, we could think of the boundary of the natural collar as being the geodesic itself in the case of geodesic boundary.
%By convention we define the $\infty$-th natural collar of a cone point to be the cone point itself.

Also note that by definition
$$\CC_{k+1}(\delta) \subset \CC_k(\delta)$$
as long as $\CC_{k+1}(\delta)$ and $ \CC_k(\delta)$ are defined.

The following proposition, crucial in our approach, really captures the relationship between the geometry of the natural collars and the dynamics of geodesic behavior which enter them. 

\begin{proposition}\label{prop:naturalk}
A complete geodesic path enters the $k$-th natural collar if and only if it forms a loop that turns $k$ times around the boundary component.
\end{proposition}

By a loop that turns $k$ times around the boundary component we mean a geodesic loop (based in a point) that is homotopic to the $k$-th power of a simple loop around the boundary component. Note we are not concerned about orientation so this is well defined. Further note that the loop is not necessarily contained inside the collar.

\begin{figure}[h]
%\ShowGrid
\leavevmode \SetLabels
\L(.235*.44) $\delta$\\%
\endSetLabels
\begin{center}
\AffixLabels{\centerline{\includegraphics[width=8cm]{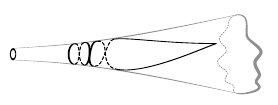}}}
\vspace{-24pt}
\end{center}
\caption{A loop that turns $4$ times around the boundary $\delta$}
\label{fig:Loop}
\end{figure}

\begin{proof}
We prove the proposition when the boundary element is a simple closed geodesic. The proof is identical in the other cases. The basic idea is to understand how deep inside the collar one can go before forming a loop around the boundary curve.

Consider such a loop $\gamma$ and compute its distance to the boundary curve. This distance $d$ is a function of the length of the loop and the length of the boundary geodesic $\delta$. This can be shown by a straightforward hyperbolic trigonometry computation.
\begin{figure}[H]
%\ShowGrid
\leavevmode \SetLabels
\L(.48*-0.14) $k\ell(\delta)$\\%
\L(.48*.3) $d$\\%
\L(.48*.77) $\ell(\gamma)$\\%
\endSetLabels
\begin{center}
\AffixLabels{\centerline{\includegraphics[width=10cm]{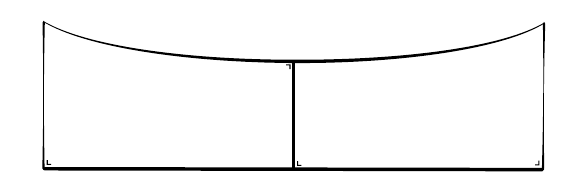}}}
\vspace{-24pt}
\end{center}
\caption{A quadrilateral obtained by unwinding a loop that turns $k$ times around the boundary $\delta$}
\label{fig:Loopunwound}
\end{figure}

An exact formula that relates these quantities is 
$$
\cosh(d)=\tanh\left( \frac{\ell(\gamma)}{2}\right) \coth\left( \frac{k\ell(\delta)}{2}\right) .
$$
The distance increases as a function of the length of the loop and the extremal case is a bi-infinite loop with a base point at infinity which lies exactly at distance
$$
\arccosh\left( \coth\left( \frac{k\ell(\delta)}{2}\right) \right)=\arcsinh\left( \frac{1}{\sinh\left(\frac{k\ell(\delta)}{2}\right)}\right)
$$
from the boundary curve.

We now pass to the converse statement, namely that a complete geodesic that enters the natural collar of $\delta$ contains a subloop that wraps $k$ times around $\delta$. 

Suppose the surface has a complete metric (so it has no cone points, and if there are boundary curves, we add the infinite funnels to complete it). In this case there is a converse to the above statement, namely that any complete geodesic that comes within distance 
$ \arcsinh\left( \frac{1}{\sinh\left(\frac{k \ell(\delta)}{2}\right)}\right)$ of the boundary curve (but does not hit it), contains a loop that wraps $k$ times around the boundary element. To see this, exponentiate forwards and backwards from the point closest to the boundary curve to obtain a geodesic loop entirely contained in the extremal ideal geodesic loop described above, and which must wrap around the base geodesic $k$ times. 

The same proof can be generalized to work for surfaces with cone points, provided the cone points all have angle $\leq \pi$, as we now show. The angle condition is essential for this to work. 

Let $X$ be a cone surface to which we've added infinite funnels to any boundary geodesics it might have. It has a universal cover $\tilde{X}$ which for non cone points is locally isometric to $\Hyp$, and for cone points is a locally isometric to a cone point of the corresponding angle. Provided $X$ has enough topology, $\tilde{X}$ is quasi-isometric to $\Hyp$, but note that if for instance $X$ is (topologically) a sphere with three cone points, then $\tilde{X} = X$ and if $X$ is a torus with cone points, then $\tilde{X}$ is quasi-isometric to the Euclidean plane. This just depends on the underlying conformal structure of $X$ with its cone points filled in. For the sake of simplicity, we can think of $\tilde{X}$ as being a "bumpy" hyperbolic plane, but this doesn't really play a part in the proof.

Let $\delta$ be a simple closed geodesic boundary and suppose that a complete geodesic $\gamma$ has entered its $k$-natural collar. We lift a $k$-fold copy of $\delta$ to $\tilde{X}$ to obtain the following setup (see Figure \ref{fig:bumpy1}) where $d$ is the distance between $\gamma$ and $\delta$.

\begin{figure}[H]
%\ShowGrid
\leavevmode \SetLabels
\L(.48*1.02) $k\ell(\delta)$\\%
\L(.26*.72) $d$\\%
\L(.73*.72) $d$\\%
\L(.29*.22) $\gamma_1$\\%
\L(.688*.22) $\gamma_2$\\%
\endSetLabels
\begin{center}
\AffixLabels{\centerline{\includegraphics[width=8cm]{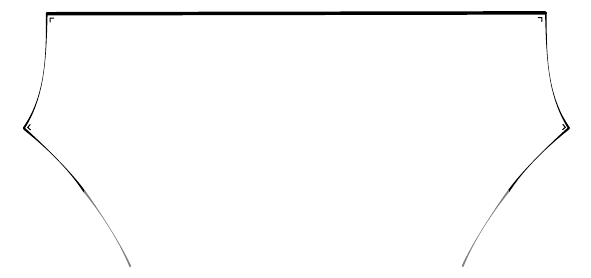}}}
\vspace{-24pt}
\end{center}
\caption{Lifting to $\tilde{X}$}
\label{fig:bumpy1}
\end{figure}
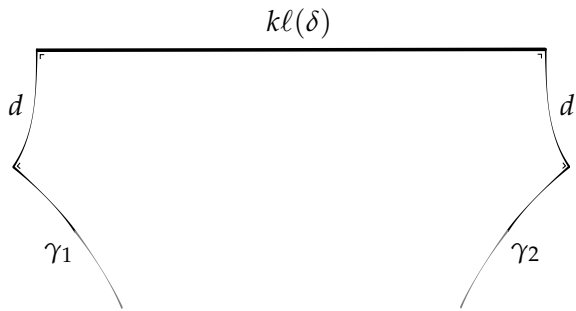

We look at the two lifts of $\gamma$ coming from exponentiating to the right and to the left from the closest point to $\delta$. We will now prove that these two lifts, $\gamma_1$ and $\gamma_2$, intersect and, in $\tilde{X}$, form a proper hyperbolic pentagon with the $k$-fold lift of $\delta$, and the two lifts of the distance path between $\gamma$ and $\delta$. In particular, we will show that this pentagon doesn't have any cone points inside. We argue by contradiction. 

Suppose this is not the case: that is, either $\gamma_1$ and $\gamma_2$ do not intersect, or, if they do, that they enclose some number of cone points (hence our pentagon has interior cone points). In both cases we consider the free homotopy class of paths between $\gamma_1$ and $\gamma_2$ which is freely homotopic to the path that leaves from $\gamma_1$, follows the distance path of length $x$ (slightly to the right) until it hits the lift of $\delta$, then follows the lift of $\delta$ before following the second lift of the distance path to $\gamma_2$ (see Figure \ref{fig:bumpy2}). There is a geodesic minimizer among such paths in the free homotopy class (with endpoints gliding on $\gamma_1$ and $\gamma_2$), perpendicular in its endpoints.

\begin{figure}[H]
%\ShowGrid
\leavevmode \SetLabels
\L(.48*1.02) $k\ell(\delta)$\\%
\L(.26*.72) $d$\\%
\L(.73*.72) $d$\\%
\L(.29*.22) $\gamma_1$\\%
\L(.688*.22) $\gamma_2$\\%
\endSetLabels
\begin{center}
\AffixLabels{\centerline{\includegraphics[width=8cm]{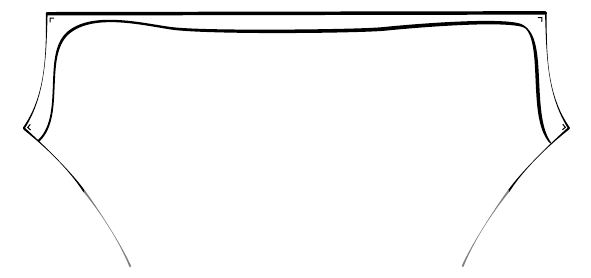}}}
\vspace{-24pt}
\end{center}
\caption{The homotopy class of path between $\gamma_1$ and $\gamma_2$}
\label{fig:bumpy2}
\end{figure}
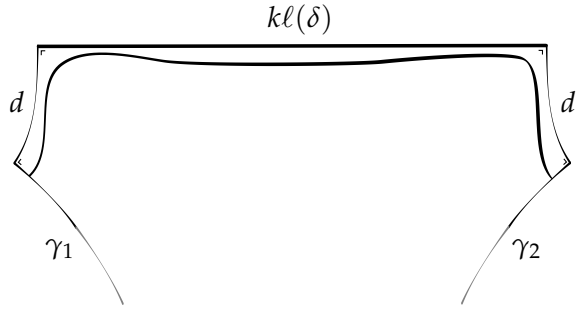

Note that the homotopy class we are considering is in the underlying surface with the cone points marked, that is we are {\it not} letting the homotopy cross any cone points. As we are not necessarily in $\Hyp$, as there are possibly cone points, we might be worried that such a geodesic minimizer gets "stuck" on a cone point, but this is where the angle condition on the cone points comes into play. 

Here is the key observation: at any point, a geodesic minimizer cannot have angle less than $\pi$ on {\it both} of its sides. Now as the cone points all have angle at most $\pi$, this means that the geodesic minimizer cannot pass through a cone point. 

Thus we get a smooth geodesic representative in this homotopy class, and we can conclude that, together with the other paths, it encloses a hyperbolic right angled hexagon (see Figure \ref{fig:bumpy3}). 

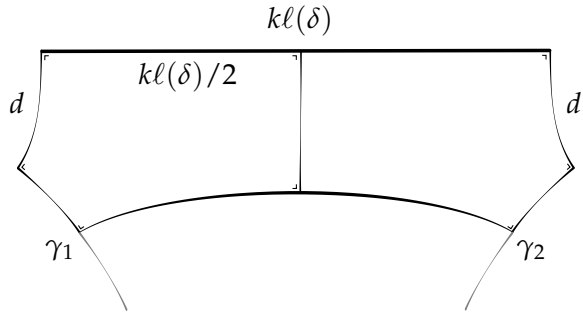
\begin{figure}[H]
%\ShowGrid
\leavevmode \SetLabels
\L(.48*1.02) $k\ell(\delta)$\\%
\L(.37*.82) $k\ell(\delta)/2$\\%
\L(.26*.72) $d$\\%
\L(.735*.72) $d$\\%
\L(.29*.22) $\gamma_1$\\%
\L(.692*.22) $\gamma_2$\\%
\endSetLabels
\begin{center}
\AffixLabels{\centerline{\includegraphics[width=8cm]{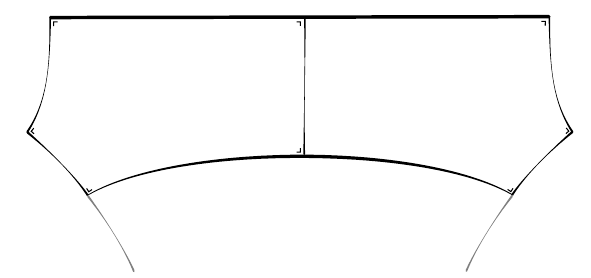}}}
\vspace{-24pt}
\end{center}
\caption{The embedded hexagon and the two isometric pentagons it splits into}
\label{fig:bumpy3}
\end{figure}

We can now argue by hyperbolic trigonometry in the hexagon. As two non consecutive sides are of length $x$, we can break into two pentagons as in Figure \ref{fig:bumpy3} and from this we see that $x$ must satisfy the relation 
$$
\sinh(d) \sinh\left(\frac{k \ell(\delta)}{2}\right) \geq 1
$$
which is impossible, because we have supposed that $\gamma$ has entered the $k$-th natural collar and hence $x$ is less than the width of the $k$-th natural collar. 

The case where $\delta$ is a cone point angle works in the same way, only the conclusion at the end involves a different hyperbolic trigonometry formula.

\begin{figure}[htbp]
\begin{center}
\includegraphics[width=10cm]{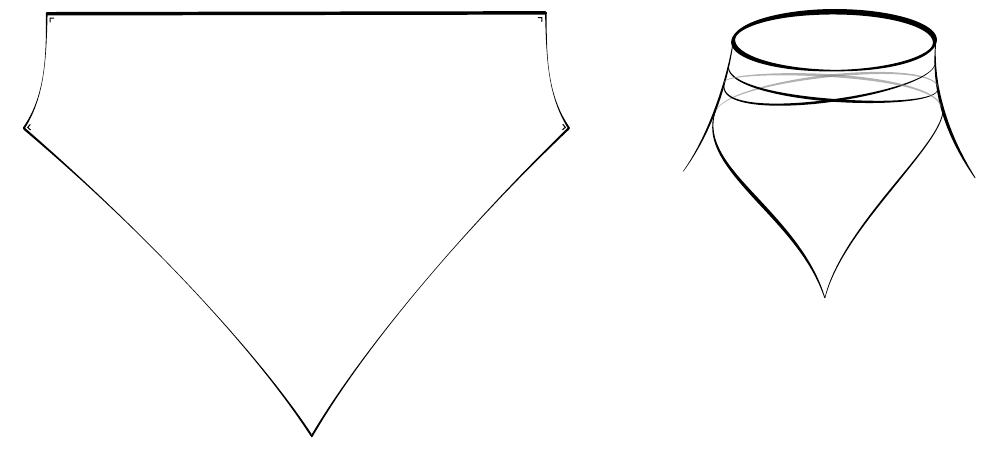}
\caption{The pentagon embedded in $\tilde{X}$ and the projected loop on $X$ (in this example $k=2$)}
\label{fig:bumpy4}
\end{center}
\end{figure}

This establishes that $\tilde{X}$ contains a hyperbolic pentagon as previously claimed. Now we can conclude by projecting the pentagon to $X$ to obtain a geodesic loop that wraps $k$ times around $\delta$ (see Figure \ref{fig:bumpy4}).
\end{proof}

The region disjoint from a set of natural collars we call the {\it concave core} of $X$. This notion is always relative to a choice of natural collars, thus an attribution of integers to each boundary. More precisely, let $X$ have boundary elements $\delta_1,\hdots,\delta_n$, and let $\vec{k}=(k_1,\hdots,k_n)$ be a grading. Recall that each $k_i \in \N \cup \infty$.

\begin{definition}\label{def:concavecore}
The concave core of $X$, relative to a choice of integer for each boundary component, is the set 
%$$V_{\vec{k}}(X) := X \setminus {\mbox{ interior}}\{\CC_{k_1}(\delta_1) \cup \hdots \cup \CC_{k_n}(\delta_n)\}$$
$$V_{\vec{k}}(X)
:= X \setminus \{\CC_{k_1}(\delta_1) \cup \hdots \cup \CC_{k_n}(\delta_n)\}$$
provided $\CC_{k_i}(\delta_i)$ are all well defined. 
\end{definition}

Note that the well defined problem above only comes into play when you have cone angles. More specifically, given a cone point surface, each cone point boundary of angle $\theta$ has a {\it maximal} grade: this is the largest integer $k$ such that $k \theta \leq \pi$. As such, any $X$ has a maximal concave core.

Note that $V_{\vec{k}}(X)$ is always connected and closed and if $\vec{k} \geq \vec{l}$ (with respect to the natural partial order) then $V_{\vec{k}}(X)\subset V_{\vec{l}}(X)$ when both are defined. With the notation $\vec{\infty}= \overbrace{(\infty,\hdots,\infty)}^{n}$, we have $V_{\vec{\infty}}(X) = X$ (which only makes sense when $X$ has no cone points). Said differently, a surface has no cone points if and only if it is equal to its maximal concave core. 

\subsection{Subsets of orthogeodesics}
\label{subsec: subsets of orthogeodesics}

This characterization of geodesic behavior in the neighborhood of boundary elements allows for a classification of orthogeodesics. At both ends, orthogeodesics leave from the boundary so traverse all natural collars of these boundary elements orthogonally. In the mid portion they might re-enter natural collars before leaving again. By the previous proposition, how deep they go inside depends essentially on how many times they wrap around the corresponding boundary element. In particular, besides its ends, an orthogeodesic might not entirely live in a given concave core but it will live in all but a finite number of them. We define the {\it depth} of an orthogeodesic $\mu$ relative to a boundary component $\delta$ to be the smallest integer $k$ such that $\mu$ does not intersect $\CC_k(\delta)$ outside of its ends.

We denote by $\OO_{\vec{k}}(X)$ the set of orthogeodesics of $V_{\vec{k}}(X)$, and by $\overline{\OO}_{\vec{k}}(X)$ their unoriented counterparts. Note that for an orthogeodesic $\mu\in \OO_{\vec{k}}(X)$, there is a well-defined orthogeodesic $\eta$ of $\OO(X)$ such that $\mu$ is the restriction of $\eta$ to $V_{\vec{k}}(X)$. In this manner, we can think of $\mu\in \OO_{\vec{k}}(X)$ as a subset of $\OO(X)$. This point of view is useful when we think of these sets as index sets. More generally we have
$$
 \OO_{\vec{k}}(X) \subset \OO_{\vec{l}}(X)
 $$
 if $k_i\leq l_i$ for all $i$ when these are defined, when thought of as abstract sets, even though geometric quantities related to an orthogeodesic, such as their length, differ depending on what concave core it is considered on. Finally, observe that the subsets $\OO_{\vec{k}}(X)$ form a natural exhaustion of $\OO(X)$.

\section{Primality of orthogeodesics}\label{sec:prime}

Here we use the immersed pair of pants point of view to define a notion of primality for orthogeodesics. We begin with a first case which illustrates the more general setup nicely. For technical reasons that will become apparent in what follows, we require that $X$ not be (topologically) a pair of pants. The notation $\vec{1}$ will be used for $\underbrace{(1,\hdots,1)}_{n}$.

\subsection{The set of $\vec{1}$-prime orthogeodesics}

We begin by considering the smallest sets of orthogeodesics $\OO_{\vec{1}}(X)$ and the associated immersed pairs of pants $\{P_\mu \in \PP(X) \mid \mu \in \OO_{\vec{1}}(X)\} $ which we denote $\PP_{\vec{1}}(X)$. The boundary curves of the pair of pants $\gamma \in \partial\PP_{\vec{1}}(X)$ may or may not live in the concave core $V_{\vec{1}}(X)$. In the sequel we will show that it lives in the concave core for special orthogeodesics (see Remark \ref{rem:concaveprime}). 

Consider an element $\gamma \in \partial\PP_{\vec{1}}(X)$ and its corresponding orthogeodesic $\mu$ between boundary elements $\alpha$ and $\beta$ of $V_{\vec{1}}(X)$ (not necessarily distinct). It is the boundary of an immersed pair of pants $P$ coming from an immersion
$$\varphi: \tilde{P}\to X$$
We denote by $\tilde{\mu} := \varphi^{-1}(\mu)$ and $\tilde{\gamma} := \varphi^{-1}(\gamma)$.

\begin{figure}[h]
%\ShowGrid
\leavevmode \SetLabels
\L(.15*0.66) $\tilde{\alpha}$\\%
\L(.48*0.66) $\tilde{\beta}$\\%
\L(0.31*0.42) $\tilde{\mu}$\\%
\L(.58*.9) $\alpha$\\%
\L(.868*.83) $\beta$\\%
\L(.68*.53) $\mu$\\%
\L(.543*0.51) $\varphi$\\%
\endSetLabels
\begin{center}
\AffixLabels{\centerline{\includegraphics[width=12cm]{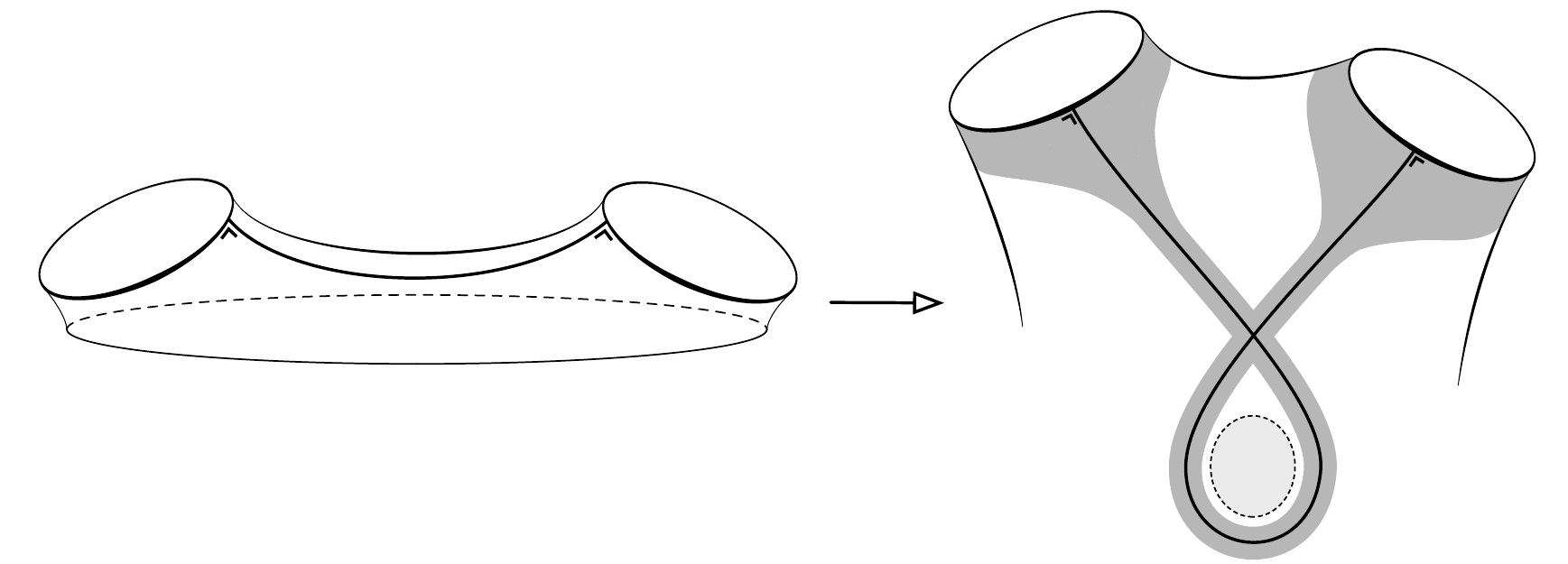}}}
\vspace{-30pt}
\end{center}
\caption{The immersion $\varphi$ of a pair of pants. The immersed pants is shaded.}
\label{fig:immersion}
\end{figure}
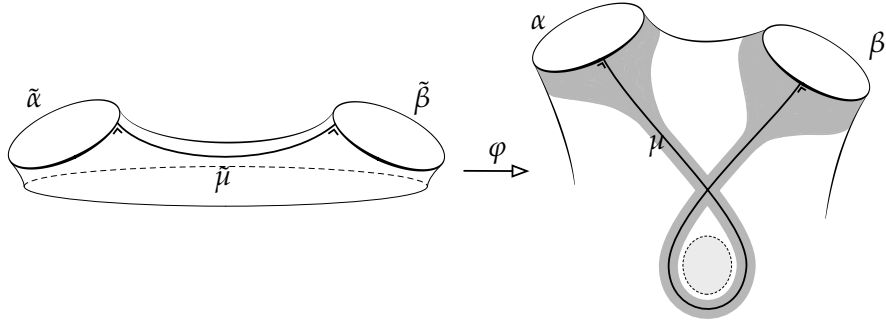

We also denote $\tilde{\alpha}:= \varphi^{-1}(\alpha)$ and $\tilde{\beta}:= \varphi^{-1}(\beta)$. 

Note that boundary curves of $\tilde{P}$ are $\tilde{\alpha}$, $\tilde{\beta}$ and $\tilde{\gamma}$. The orthogeodesic $\tilde{\mu}$ is a simple orthogeodesic of $\tilde{P}$ between $\tilde{\alpha}$ and $\tilde{\beta}$. Note that any other orthogeodesic $\tilde{\eta}$ with both ends lying on either $\tilde{\alpha}$ or $\tilde{\beta}$ corresponds to an orthogeodesic $\eta \in \OO_{\vec{1}}(X)$. The corresponding immersed pair of pants $P_\eta$ is, as a set, contained inside $P_\mu$. This allows us to define a partial order on $\PP_{\vec{1}}(X)$ given by
$$
P_\eta < P_\mu {\mbox{ if }} P_\eta \subset P_\mu
$$
This leads to the notion of maximality. 

\begin{definition} An element of $\PP_{\vec{1}}(X)$ is said to be $\vec{1}$-prime if it is maximal with respect to this partial order.
\end{definition}

We also get a partial order on unoriented orthogeodesics. If for $\mu,\eta \in \OO_{\vec{1}}(X)$ we denote the corresponding unoriented orthogeodesics by $\bar{\mu}$ and $\bar{\eta}$, we have 
$$
\bar{\mu} < \bar{\eta} \mbox{ if } P_\eta \subset P_\mu.
$$
If we want to define a partial order on $\OO_{\vec{1}}(X)$, we need to be a little bit more careful. We have
$$
 \eta < \mu {\mbox{ if $\eta$ and $\mu$ start on the same boundary element and if }} P_\eta \subset P_\mu.
$$

Via the one-to-one correspondence, primality then extends to unoriented orthogeodesics $\overline{\OO}_{\vec{1}}(X)$. As $\OO_{\vec{1}}(X)$ is a subset of $\OO(X)$, we get a partial order on all orthogeodesics and hence there is a notion of $\vec{1}$-primality for an element of $\OO(X)$. 

The following proposition about primality will be proved later in a more general context.

\begin{proposition}\label{prop:prime1}
If $\bar{\eta}\in \overline{\OO}_{\vec{1}}(X)$ is not $\vec{1}$-prime then there exists a unique $\vec{1}$-prime $\bar{\mu}$ such that $\bar{\eta} < \bar{\mu}$. 
\end{proposition}

The result about unicity does {\it not} necessarily hold for oriented orthogeodesics, but to what extent it fails is easy to understand. The problem only appears when we have oriented orthogeodesics leaving and ending on the same boundary element. In this case, there are two prime orthogeodesics: the orthogeodesic and the one with opposite orientation.

We denote the set of $\vec{1}$-prime immersed pair of pants, resp. orthogeodesics, by $\PP'_{\vec{1}}(X)$, resp. $\OO'_{\vec{1}}(X)$. We give three examples of orthogeodesics on a surface with a single boundary component in Figure \ref{fig:primeornot}, and one of them is not $\vec{1}$-prime. 

\begin{figure}[htbp]
\begin{center}
\includegraphics[width=10cm]{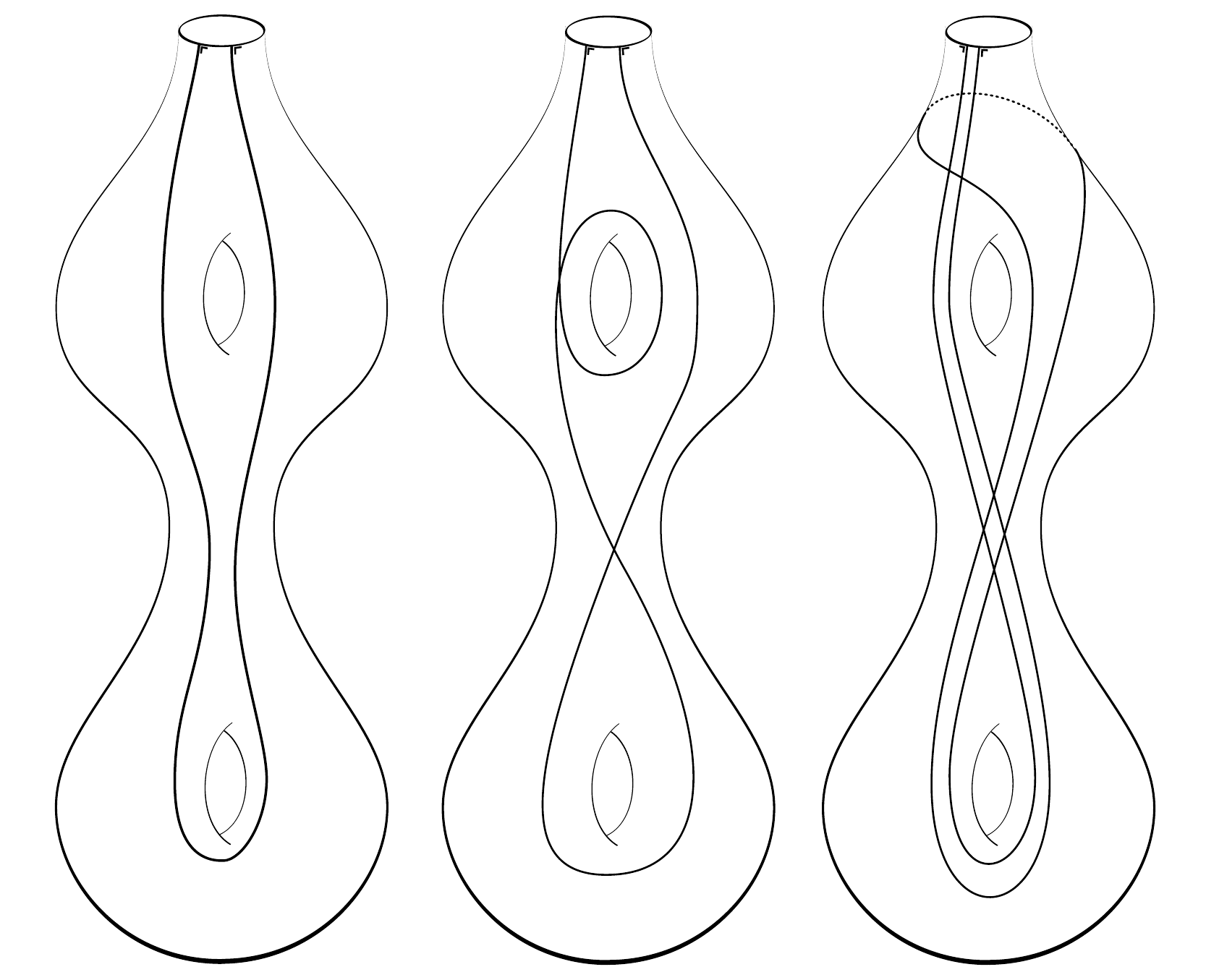}
\caption{Which orthogeodesic is not $\vec{1}$-prime?}
\label{fig:primeornot}
\end{center}
\end{figure}

The point to be stressed here is that with just this definition in hand, both proving Proposition \ref{prop:prime1} and figuring out if a given orthogeodesic is prime, is not a priori obvious. For this we introduce a tool. 

\subsection{Model surfaces, take $\vec{1}$}

Given a topological surface $\Sigma=\Sigma_{g,n}$, homeomorphic to $X$, we define a $\vec{1}$-model surface $M_{\vec{1}}$ to be a hyperbolic surface with all boundary elements realized as cone points of angle $\pi$. In this case we ask $\Sigma$ be of negative Euler characteristic and not be homeomorphic to a pair of pants, it is always possible to find such a metric unless $(g,n)=(0,4)$. We will return to this case in the sequel. Note that the choice of such a specific hyperbolic metric is irrelevant, as will be made precise in Proposition \ref{prop:whocares}.

Orthogeodesics of $X$ correspond to paths between boundary in $\Sigma$ up to homotopy relative to boundary. These paths can be realized on $M_{\vec{1}}$ and in fact have unique geodesic minimizers. These minimizers may however ``collapse" in the cone point singularities and will in general consist in a collection of geodesics between the $\pi$ cone point singularities of $M_{\vec{1}}$. Those that only pass through a cone point singularity at end points are special. These will be called {\it properly immersed} and are intimately related to $\vec{1}$-prime orthogeodesics.

\begin{proposition}\label{prop:model1}
An immersed pair of pants (or orthogeodesic) is $\vec{1}$-prime if and only if the corresponding geodesic on $M_{\vec{1}}$ is properly immersed.
\end{proposition}

\begin{proof}
The main point is to observe what happens to immersed pairs of pants on the model surface. 

Consider a pair of pants with two boundary cone points of angle $\theta <\pi$ and a third boundary element a simple closed geodesic $\gamma$. Now if we let $\theta \to \pi$, with the length of $\gamma$ fixed, in the limit the entire pair of pants is the limit of $\gamma$.

\begin{figure}[htbp]
\begin{center}
\includegraphics[width=12cm]{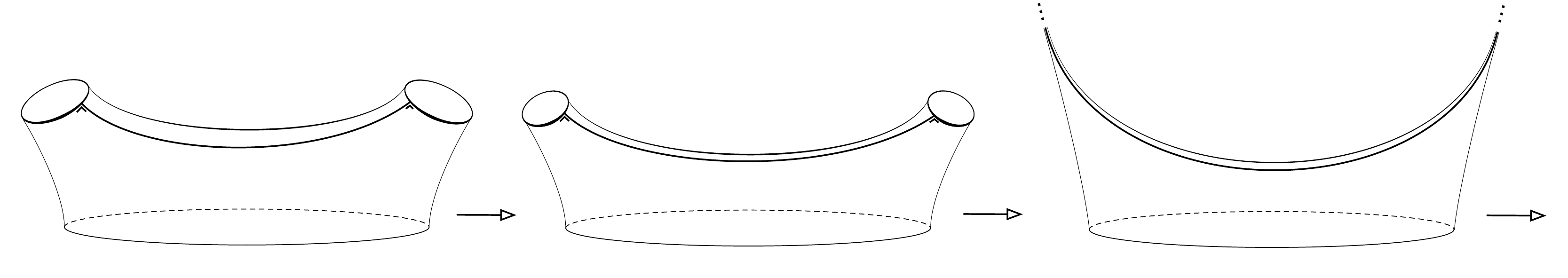}
%\caption{default}
\vspace{-20pt}
\end{center}
\end{figure}
\begin{figure}[htbp]
\begin{center}
\includegraphics[width=14cm]{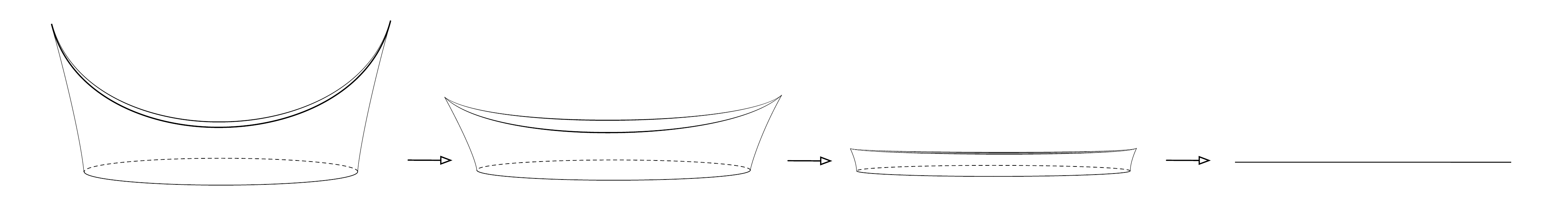}
\caption{A pair of pants: from geodesic boundary, to cusp boundary, to cone point boundary, to a degenerate pair of pants}
\label{fig:collapse}
\end{center}
\end{figure}

With the point of view of immersed pair of pants on $X$, consider $\mu$ as the image by $\varphi$ of $\tilde{\mu} \subset \tilde{P}_{\mu}$. If $\eta < \mu$ then $\eta$ can be realized as the image by $\varphi$ of an orthogeodesic $\tilde{\eta} \subset \tilde{P}_{\mu}$. 

When we look at this setup for orthogeodesics on $M_{\vec{1}}$, then $\tilde{P}_\mu$ has collapsed to a single geodesic segment (in fact $\tilde{P}_\mu = \tilde\mu$) and so $\eta$ can be nothing else but a concatenation of at least two copies of $\mu$. In particular, it is not $\vec{1}$-prime which proves the proposition.
\end{proof}

\subsection{Primality in full generality}\label{subsec:Primality in full generality}

We want to generalize the notion of primality. We shall see how the model surface point of view generalizes very nicely to the general setup and, because it will be equivalent, it could be taken as a definition. However, as before, we begin with the topological setup first. 

Consider an element $\mu \in \OO_{\vec{k}}(X)$ between boundary elements $\alpha$ and $\beta$ of $V_{\vec{k}}(X)$ (not necessarily distinct). Associated to $\mu$ is a curve $\gamma \in \partial\PP(X)$ and $P$ the corresponding immersed pair of pants
$$\varphi: \tilde{P}\to X.$$
We use the same notation as previously, namely $\tilde{\mu} := \varphi^{-1}(\mu)$, $\tilde{\gamma} := \varphi^{-1}(\gamma)$, $\tilde{\alpha}:= \varphi^{-1}(\alpha)$ and $\tilde{\beta}:= \varphi^{-1}(\beta)$.

To the boundary elements $\alpha$ and $\beta$ we've associated integers from the grading $\vec{k}$ corresponding to their order. We denote these integers $k_\alpha$ and $k_\beta$ (their {\it grades}). 

Note that before the partial order on $\PP_{\vec{1}}(X)$ came from the inclusion of immersed pants. These immersed pants were, as point sets, exactly the closure of the the union of all connected components of $X\setminus \gamma$ which intersect $\mu$. For an orientation of $\alpha, \beta$ and $\mu$, the curve $\gamma$ is homotopic to $\alpha*\mu*\beta*\mu^{-1}$. This generalizes as follows. 

We define the homotopy class
$$
[\gamma_\mu]: = [\alpha^{k_\alpha}*\mu*\beta^{k_\beta}*\mu^{-1}]
$$
and $\gamma_\mu $ to be the corresponding closed geodesic. 

With the point of the view of the immersion we have the following setup. We define $\tilde{\gamma}_\mu$ to be the geodesic in the homotopy class
$$
 [\tilde{\alpha}^{k_\alpha}*\tilde{\mu}*\tilde{\beta}^{k_\beta}*\tilde{\mu}^{-1}]
$$
and 
$$
\varphi(\tilde{\gamma}_\mu) = \gamma_\mu
$$
For $\gamma_\mu$, we denote by $Q_\mu$ the immersed pair of pants bounded by $\alpha$, $\beta$ and $\gamma_\mu$. Note that, as before, $Q_\mu$ is the closure of the union of the the connected components of $X\setminus \gamma_\mu$ that intersect $\mu$. Note that if $k_\alpha = k_\beta =1$, then $Q_\mu= P_\mu$. Otherwise $Q_\mu \subset P_\mu$ and the containment is strict. 

\begin{figure}[htbp]
%\ShowGrid
\leavevmode \SetLabels
\L(.2*0.66) $\tilde{\tilde{\alpha}}$\\%
\L(.35*0.66) $\tilde{\tilde{\beta}}$\\%
\L(.13*0.17) ${\tilde{\tilde{\gamma}}}_\mu$\\%
\L(0.29*0.35) $\tilde{\tilde{\mu}}$\\%
\L(.643*.46) $\tilde{\alpha}$\\%
\L(.827*.455) $\tilde{\beta}$\\%
\L(.734*.365) $\tilde{\mu}$\\%
\L(.755*0.66) $\tilde{\gamma}_\mu$\\%
\L(.53*0.55) $\psi$\\%
\endSetLabels
\begin{center}
\AffixLabels{\centerline{\includegraphics[width=12cm]{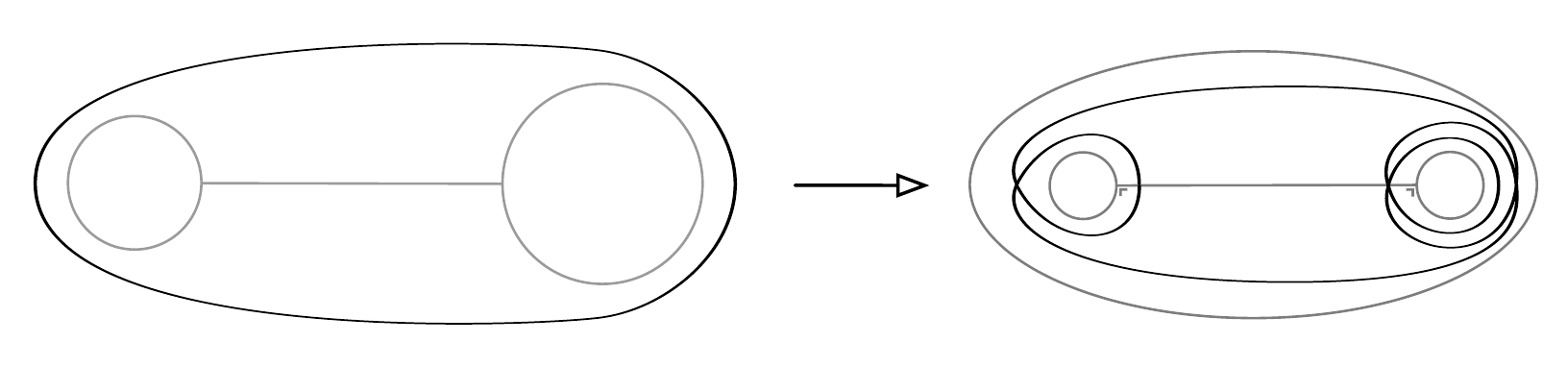}}}
\vspace{-30pt}
\end{center}
\caption{The immersion $\psi$ from $\tilde{\tilde{P}}_\mu$ to $\tilde{P}_\mu$ where $k_\alpha = 2$ and $k_\beta=3$.}
\label{fig:preimmersion}
\end{figure}
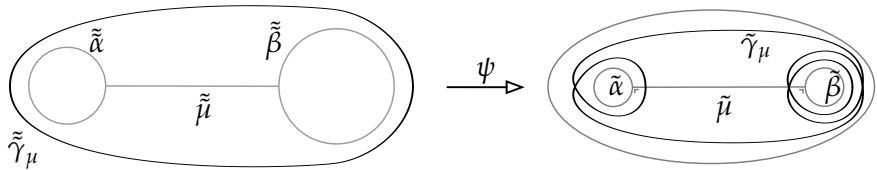

The set $Q_\mu$ is also an immersed pair of pants. To see this, it is convenient to see $Q_\mu$ as a the result of a double immersion. We take a pair of pants $\tilde{\tilde{P}}$ with two boundary lengths $k_\alpha \ell(\alpha)$ and $k_\beta \ell(\beta)$ and a simple orthogeodesic of the length of $\mu$ between them. We denote the two boundary components by $\tilde{\tilde{\alpha}}$ and $\tilde{\tilde{\beta}}$ and the third boundary by $\tilde{\tilde{\gamma}}_\mu$. The pair of pants $\tilde{\tilde{P}}_\mu$ is immersed in ${\tilde{P}}_\mu$ via a map $\psi:\tilde{\tilde{P}}_\mu \to {\tilde{P}}_\mu$ as portrayed in Figure \ref{fig:preimmersion}. We define $\tilde{\gamma}_\mu$ to be $\psi\left( \tilde{\tilde{\gamma}}_\mu\right)$. 

Note that the $\psi$ acts $k_\alpha$ to $1$ on $\tilde{\tilde{\alpha}}$ and $k_\beta$ to $1$ on $\tilde{\tilde{\beta}}$. Now applying $\varphi$ gives an immersion $\varphi\circ \psi: \tilde{\tilde{P_\mu}} \to Q_\mu$ (see Figure \ref{fig:doubleimmersion} for an illustration of $\varphi$ and the resulting curve $\gamma_\mu$). The subset corresponding to the preimages of $Q_\mu$ via $\varphi$ we denote by $\tilde{Q}_\mu$. So $\tilde{Q}_\mu \subset {\tilde{P}}_\mu$. 

The set of immersed pairs of pants $Q_\mu$ defined as above is denoted by $\PP_{\vec{k}}(X)$ and the set of curves $\gamma_\mu$ we denote by $\partial\PP_{\vec{k}}(X)$. 

\begin{figure}[H]
%\ShowGrid
\leavevmode \SetLabels
\L(.32*0.58) $\tilde{\gamma}_\mu$\\%$\\%
\L(.662*0.693) $\gamma_\mu$\\%$\\%
\L(.489*0.54) $\varphi$\\%
\endSetLabels
\begin{center}
\AffixLabels{\centerline{\includegraphics[width=12cm]{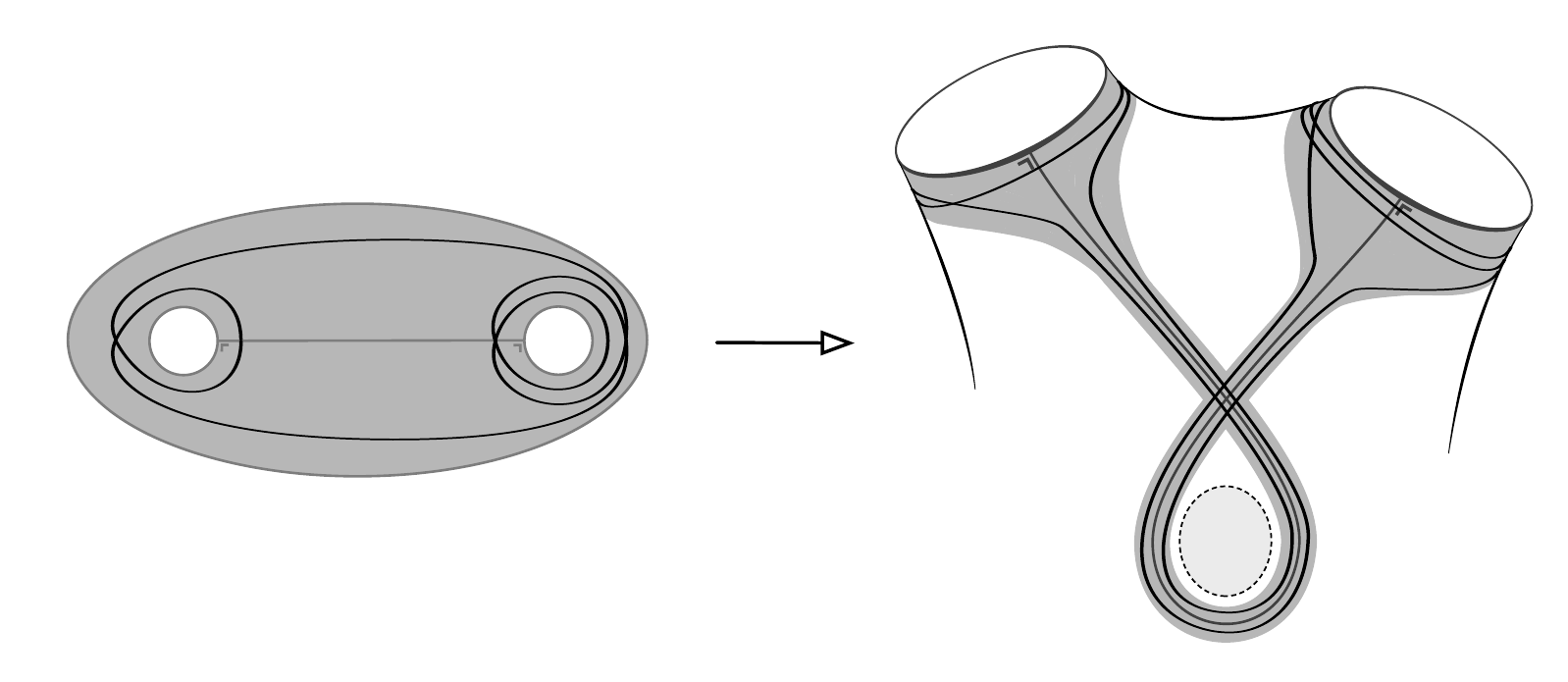}}}
\vspace{-30pt}
\end{center}
\caption{The immersion $\varphi$ from $\tilde{P}_\mu$ to $P_\mu$ (on the right and shaded). The figure is topological and both $\gamma_\mu$ and its preimage $\tilde{\gamma}_\mu$ are illustrated.}
\label{fig:doubleimmersion}
\end{figure}

This defines a partial order $<_{\vec{k}}$ on $\PP_{\vec{k}}(X)$:
$$
 Q_\eta <_{\vec{k}} Q_ \mu {\mbox{ if }} Q_\eta \subset Q_\mu
$$
and then a partial order on $\OO_{\vec{k}}(X)$ given by
$$
\eta <_{\vec{k}} \mu {\mbox{ if }} Q_\eta \subset Q_\mu.
$$
Note that this partial order is {\it not} the same partial order as before and so to avoid any possibility of confusion we have indexed it by ${\vec{k}}$. 

Observe that if $\eta <_{\vec{1}} \mu$ then $\eta <_{\vec{k}} \mu$ for any $\vec{k}$. (This is straightforward with the immersed pants point of view.) We can now define primality in general.
 
\begin{definition} An element of $\PP_{\vec{k}}(X)$, resp. $\OO_{\vec{k}}(X)$, is said to be $\vec{k}$-prime if it is maximal with respect to this partial order.
\end{definition}

Again if we think of the sets $\OO_{\vec{k}}(X)$ as subsets of $\OO(X)$, the above definition makes sense for all orthogeodesics and immersed pairs of pants. In particular, any two distinct elements in $\OO(X)$ may or may not be related by a (finite) number of partial orders.

\begin{definition} An element of $\OO(X)$ (resp. $\overline{\OO}(X)$) is said to be $\vec{k}$-prime if it belongs to $\OO_{\vec{k}}(X)$ (resp. $\overline{\OO}_{\vec{k}}(X)$) and it is maximal with respect to the partial order $<_{\vec{k}}$.
\end{definition}

In particular, for given $\vec{k}$ and $\vec{l} $, we can talk about the $\vec{k}$-primality of an element $\mu \in \OO_{\vec{l}}(X)$. We denote the set of $\vec{k}$-prime elements $\PP'_{\vec{k}}(X)$ and $\OO'_{\vec{k}}(X)$. 

\begin{remark}\label{rem:onetorulethemall} The notion of primality for all the different sets $\OO(X)$, $\OO_{\vec{k}}(X)$, $\overline{\OO}(X)$, $\overline{\OO}_{\vec{k}}(X)$, $\PP_{\vec{k}}(X)$, might seem confusing at first, but in fact they are all the same. If an element of one set is prime, then the corresponding element of the other set is prime too, hence it suffices to work with one of the sets. 
\end{remark}

The following proposition generalizes Proposition \ref{prop:prime1}. 
\begin{proposition}\label{prop:uniqueprimek}
If $\eta \in \overline{\OO}_{\vec{k}}(X)$ is not $\vec{k}$-prime then there exists a unique $\vec{k}$-prime $\mu$ such that $\eta <_{\vec{k}} \mu$.\end{proposition}

By the one to one correspondence, the same proposition holds for elements of $\PP_{\vec{k}}(X)$. Because of orientation issues, the corresponding proposition is slightly more complicated for elements of $\OO_{\vec{k}}(X)$:

\begin{proposition}\label{prop:almostuniqueprimek}
If $\eta \in \OO_{\vec{k}}(X)$ is not $\vec{k}$-prime then there exists a $\vec{k}$-prime $\mu$ such that $\eta <_{\vec{k}} \mu$. If $\eta$ is not unique, then the set of $\vec{k}$-prime oriented orthogeodesics is exactly $\mu$ and $\mu^{-1}$. Furthermore, this occurs exactly when $\mu$ leaves and ends on the same boundary component. 
\end{proposition}

The proof of both propositions will use model surfaces. Why there is a difference between the two propositions is straightforward however: an oriented orthogeodesic and its opposite define the same immersed pair of pants which in turn define the partial order. Hence, the only way to differentiate between them is if they have a different boundary element they leave from.

\subsection{Model surfaces, the general case}
\label{subsec: Model surfaces}

Given a topological surface $\Sigma=\Sigma_{g,n}$, homeomorphic to $X$ and of negative Euler characteristic, we define a $\vec{k}$-model surface $M_{\vec{k}}$ to be a hyperbolic surface with all ordered boundary elements realized as cone points of angles $\frac{\pi}{k_1}, \hdots, \frac{\pi}{k_n}$ provided such a surface exists. Note that if $k_i = \infty$, then the corresponding boundary element of $M_{\vec{k}}$ is a cusp. 

\begin{proposition}\label{prop:whocares}
Let $X$ and $Y$ be a choices of model surface $M_{\vec{k}}$ and let $[c]$ be a homotopy class of path with end points on boundary elements. The unique geodesic representative of $[c]$ on $X$ is properly immersed if and only if it is properly immersed on $Y$. 
\end{proposition}

\begin{proof}
Note that the unique geodesic representative of $[c]$ might not actually be homotopic to $c$, as loops might degenerate in the cone points. More precisely, if $c$ has a sub-loop that winds $k$ times around a boundary of index $k_i$ with $k_i\leq k$, it will degenerate to a geodesic that goes into the corresponding cone point, and then leaves again. In fact that is the only thing that can prevent the geodesic representative from being homotopic to $c$. As this does not depend on the choice of model surface, the result follows.
\end{proof}

We have the same setup as before: orthogeodesics of $X$ correspond to paths between boundary in $\Sigma$ up to homotopy relative to boundary. These paths can be realized on $M_{\vec{k}}$ and have unique geodesic minimizers which may however collapse in the cone point singularities. Those that only pass through cone point singularities in their two end points are again called {\it properly immersed}. In fact, we have the following. 

\begin{proposition}\label{prop:modelk}
An immersed pair of pants or orthogeodesic is $\vec{k}$-prime if and only if the corresponding geodesic on $M_{\vec{k}}$ is properly immersed. 
\end{proposition}

\begin{proof}
As before, the main point is to observe what happens from the immersed pair of pants point of view. 

Consider an orthogeodesic $\mu$ which has end points on two boundary elements $\alpha$ and $\beta$, with indices $k_\alpha$ and $k_\beta$. We want to show that on the model surface $M_{\vec{k}}$, the subsurface of bounded by $\alpha$, $\beta$ and the geodesic $\gamma_\mu$ degenerates to a geodesic. To show this we argue geometrically. 

Using the immersed pair of pants point of view, we consider a pair of pants with two boundary cone points of angles $\theta_\alpha < \frac{2\pi}{k_\alpha+1}$ and $\theta_\beta < \frac{2\pi}{k_\beta+1}$ and a third boundary element a simple closed geodesic $\tilde{\gamma}$. The curve $\tilde{\gamma}_{\mu}$ is as defined above that is in the homotopy class 
$
 [\tilde{\alpha}^{k_\alpha}*\tilde{\mu}*\tilde{\beta}^{k_\beta}*\tilde{\mu}^{-1}].
$
We now consider the limiting pair of pants obtained by letting the angles $\theta_\alpha$ and $\theta_\beta$ increase until $\frac{\pi}{k_\alpha}$ and $\frac{\pi}{k_\beta}$, while maintaining the length of $\gamma$ fixed. We now focus on the limit of the curve $\tilde{\gamma}_{\mu}$. For the same reasons as in Proposition \ref{prop:whocares}, it limits to a geodesic segment (between the two boundaries not equal to $\gamma$). The length of $\tilde{\gamma}_{\mu}$ limits to twice the length of the limiting segment.

With the point of view of immersed pair of pants on $X$, consider $\mu$ as the image by $\varphi\circ \psi$ of $\tilde{\tilde{\mu}} \subset \tilde{\tilde{P}}_{\mu}$. Now if $\eta <_{\vec{k}} \mu$ then that $\eta$ can be realized as the image by $\varphi\circ \psi$ of an orthogeodesic $\tilde{\tilde{\eta}} \subset \tilde{\tilde{P}}_{\mu}$.

When we look at this setup for orthogeodesics on $M_{\vec{k}}$, then $\tilde{\tilde{P}}_\mu$ is has collapsed to a single geodesic segment. That is, on $M_{\vec{k}}$, $Q_\mu =\mu$. As such $\eta$ can be nothing else but a concatenation of at least two copies of $\mu$ and thus it is not $\vec{k}$-prime. This proves the proposition.
\end{proof}

With this point of view, as before, Proposition \ref{prop:uniqueprimek} now follows from the uniqueness of geodesic minimizers. 

\begin{remark}\label{rem:concaveprime}
By combining the model surface point of view, and the dynamic interpretation of Proposition \ref{prop:naturalk}, there is a third possible interpretation of primality: 

{\it An orthogeodesic $\mu$ is $\vec{k}$-prime if and only if $\gamma_\mu$ is contained in the $V_{\vec{k}}(X)$ concave core.}

As we don't use this point of view, we omit the details of the proof, but essentially it goes as follows. Using the model surface point of view, it is clear that for a $\vec{k}$-prime orthogeodesic $\mu$ between $\alpha$ and $\beta$, $\gamma_\mu$ never enters the concerned natural collars of all boundary elements different from $\alpha$ and $\beta$. And a local argument (which is perhaps easier to see using the preimmersed pair of pants point of view) ensures it does not enter the natural collars of $\alpha$ and $\beta$. So $\gamma_\mu$ is contained in the concave core. Now if we have $\mu$ which is not $\vec{k}$-prime, $\gamma_\mu$ is entirely contained inside an immersed pair of pants associated to a $\vec{k}$-prime orthogeodesic between boundary curves $\alpha'$ and $\beta'$. Another local argument shows that $\gamma_\mu$ must intersect the appropriate natural collar of either $\alpha'$ or $\beta'$, and hence $\gamma_\mu$ is not contained in the $\vec{k}$ concave core. 
\end{remark}

\subsection{A prime exhaustion of $\OO(X)$}

Take an element $\mu \in \OO_{\vec{k}}(X)$ and suppose that it is not prime. Let $\mu'$ be the unique $\vec{k}$-prime such that $\mu <_{\vec{k}} \mu'$ (by Proposition \ref{prop:uniqueprimek}). The endpoints of $\mu'$ have grades $i_0$ and $i_1$ (not necessarily distinct). Let $\vec{k}^{+}$ be the element obtained by taking $\vec{k}$ by replacing $k_{i_0}$ with $k_{i_0}+1$ and $k_{i_1}$ with $k_{i_1}+1$ (and if $i_0=i_1$, only replace $k_{i_0}$ with $k_{i_0}+1$). Depending on the geometry of $X$, $V_{\vec{k}^{+}}(X)$ may or may not exist. If it does, ithe following proposition holds.

\begin{proposition}
The element $\mu\in \OO_{\vec{k}}(X)$ is $\vec{k}^{+}$-prime.
\end{proposition}

\begin{proof}
Consider the immersed pair of pants $Q_{\mu'} \in \PP_{\vec{k}}(X)$ in which $\mu$ is also contained. As $\mu$ is in $\OO _{\vec{k}}(X)$, it never wraps more than respectively $k_{i_0}-1$ and $k_{i_1}-1$ times around the corresponding boundary elements of $Q_{\mu'}$. As $\mu'$ is $\vec{k}$-prime, $Q_{\mu'}$ corresponds to a geodesic segment on $M_{\vec{k}}$, and $\mu$, as it is not $\vec{k}$-prime, corresponds to multiple copies of this geodesic segment which comes from the fact that it must wrap exactly $k_{i_0}-1$ or $k_{i_1}-1$ times around the corresponding boundary elements. However on $M_{\vec{k}^{+}}$, the angle has been decreased in the two corresponding cone points and hence $\mu$ now corresponds to a genuine properly immersed geodesic segment and is thus $\vec{k}^{+}$-prime.
\end{proof}

This implies a number of immediate corollaries.

We obtain the following weaker statement, strictly equivalent when there is only one boundary component, provided $V_{\vec{k}+ \vec{1}}(X)$ exists.

\begin{corollary}\label{cor:k+1}
An element of $\OO_{\vec{k}}(X)$ is always $(\vec{k} + \vec{1})$-prime.
\end{corollary}

Now suppose that $X$ only has cusps and geodesics as boundaries. In this case, the sets $\OO_{\vec{k}}(X)$ with $\vec{k}$ finite form an exhaustion of $\OO(X)$, and hence we also have the following. 

\begin{corollary}
The sets $\OO^{\prime} _{\vec{k}}(X)$ for $\vec{k}$ finite form an exhaustion of $\OO (X)$.
\end{corollary}

\begin{proof}
Given any element $\mu \in \OO(X)$, there exists a finite $\vec{k}$ for which $\mu \in \OO_{\vec{k}}(X)$. This is because, with the exception of its boundary ends, an orthogeodesic cannot go arbitrarily deep inside the collar of a boundary element. Now by Corollary \ref{cor:k+1}, $\mu$ is $(\vec{k} + \vec{1})$-prime. 
\end{proof}

\begin{remark} There is a way of making sense of an exhaustion if $X$ has cone points as follows. Such an $X$ has a maximal concave core, with grades $\vec{k}_{\max}$. In fact the above argument shows that the sets $\OO^{\prime} _{\vec{k}}(X)$ for an exhaustion of $\OO_{\vec{k}_{\max}}(X)$.
\end{remark}
%%%%%%Proof of the abstract identity%%%%%

\section{Proof of the abstract identity}\label{sec:ergodicity}

In this section, we prove the abstract identity. The basic set-up is as follows: Suppose that $X$ is a cone hyperbolic surface and $\vec {k}$ is a grading on $\Sigma$ which is admissible for $X$. Let $\alpha$ be a boundary component of $V:=V_{\vec{k}}(X)$ with positive measure (length). To every prime orthogeodesic $\mu \in \OO_{\vec {k}}(X)$ starting at $\alpha$ and ending at any boundary component $\beta$ (not necessarily distinct from $\alpha$), one can associate a gap in $\alpha$ containing the initial point of $\mu$ with the property that an orthoray starts at a point in this gap if and only if it stays in the pre-immersed pair of pants $P_{\vec{k}}(\mu)$ for infinite time, or is a finite arc contained in $P_{\vec{k}}(\mu)$ which ends at $\alpha$ or $\beta$ (see Claim 1 in Section 5.2). In the case where $\beta=\alpha$, there are actually two gaps, since there are two possible orientations $\mu$ and $-\mu$ which both start at $\alpha$, and which have the same pre-immersed pair of pants. We will show that these gaps are disjoint and the complementary set has measure zero. The first observation is that if an orthoray starts at a point in the complementary set, then it stays in $V$ for infinite time. Otherwise, it intersects a boundary component of $V$ in finite time, so is homotopic to some $\mu \in \OO_{\vec{k}}(X)$, and hence must lie in the gap associated to a prime orthogeodesic $\mu'$ where $\mu'\in \OO_{\vec{k}}'(X)$ is such that $\mu <\mu'$. We show in the next section that the set of initial points on $\alpha$ for orthorays which stay in $V$ for infinite time has measure zero, hence the measure of the complementary set is zero.

\subsection{Orthorays and their boundary measure}

We give a self-contained, direct proof that if $V$ has at least one boundary component with positive measure, the boundary measure of the set of orthorays that stay inside the concave core is zero. 

 Let $\alpha$ be a boundary component of $V$ with length $\lambda_{\alpha}>0$ (see remark 5.4(i) for the case where $\alpha$ is a cusp or orbifold cone point, and $\lambda_{\alpha}=0$). Identify the set of orthorays with initial point on $\alpha$ with $\alpha$. We are interested in the set $\Lambda \subset \alpha$ of orthorays that start in $\alpha$ and stay in $V$ for all positive time.

Let $\epsilon >0$ small. Consider an orthoray $\rho$ in $\Lambda$ and a fixed lift $\tilde{\rho}$ of $\rho$ to $\Bbb H$. We consider the set of geodesic rays based in $\alpha$ whose initial point is a distance less than $\epsilon$ from $\rho (0)$ along $\alpha$ which have a lift asymptotic to $\tilde{\rho}$ (see Figure \ref{fig:measure1}).

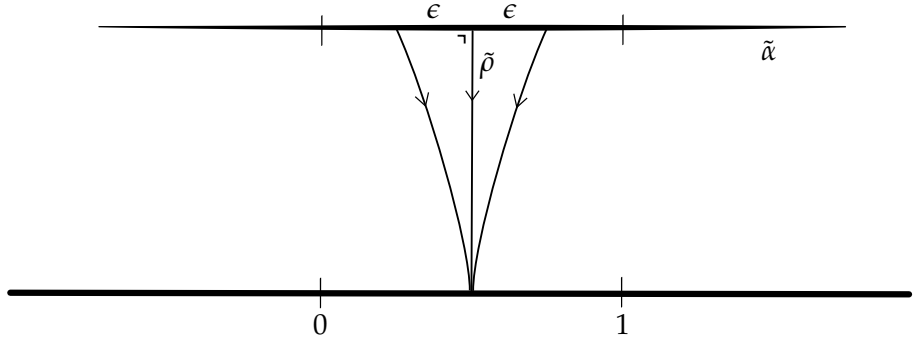
\begin{figure}[H]
%\ShowGrid
\leavevmode \SetLabels
\L(.51*0.78) $\tilde{\rho}$\\%$\\%
\L(.75*0.81) $\tilde{\alpha}$\\%$\\%
\L(.3681*0.08) $0$\\%
\L(.6257*0.08) $1$\\%
\L(.53*0.92) $\epsilon$\\%
\L(.465*0.92) $\epsilon$\\%
\endSetLabels
\begin{center}
\AffixLabels{\centerline{\includegraphics[width=14cm]{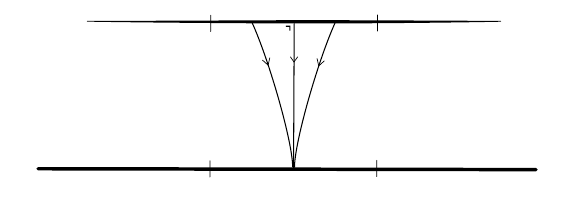}}}
\vspace{-30pt}
\end{center}
\caption{The lift of $\rho$ and its neighbors with the same limit point}
\label{fig:measure1}
\end{figure}

We denote this set of geodesic rays by $\Lambda_{\rho,\epsilon}$ and we set $\Lambda_{\epsilon}=\cup_{\rho \in \Lambda} \Lambda_{\rho, \epsilon}$.

 While the geodesic rays in the set $\Lambda_{\epsilon}$ do not necessarily stay in $V$ they do stay within a distance 
$\epsilon$ of $V$. This follows from the fact that each geodesic ray fellow travels some orthoray 
$\rho$ which stays in $V$. For $\epsilon$ small enough the geodesic rays in $\Lambda_{\epsilon}$ extend backward from the boundary component $\alpha$ to the curve a distance $\epsilon$ from 
$\alpha$. In this way the forward orbit of one of the geodesic rays can not join up with the backward orbit of another one. As a result we have:

\begin{lemma} \label{lem: distinct orthorays}
The unit tangent vectors to the geodesic rays in 
$\Lambda_{\epsilon}$ are distinct. 
\end{lemma}

For $\delta>0$, denote the set of unit tangent vectors to the geodesic rays in $\Lambda_{\epsilon}$ that flow a distance at most $\delta$ from their initial tangent vector by 
 $\Gamma_{\delta}$ (see Figure \ref{fig:measure2})).

\begin{figure}[H]
%\ShowGrid
\leavevmode \SetLabels
%\L(.51*0.78) $\tilde{\rho}$\\%$\\%
%\L(.75*0.81) $\tilde{\alpha}$\\%$\\%
%\L(.3681*0.08) $1$\\%
\L(.661*0.593) $\rho$\\%
\L(.6257*0.2) $V$\\%
\L(.335*0.535) $\epsilon$\\%
\L(.342*0.476) $\epsilon$\\%
\L(.337*0.42) $\epsilon$\\%
\L(.36*0.58) $\beta$\\%
\endSetLabels
\begin{center}
\AffixLabels{\centerline{\includegraphics[width=10cm]{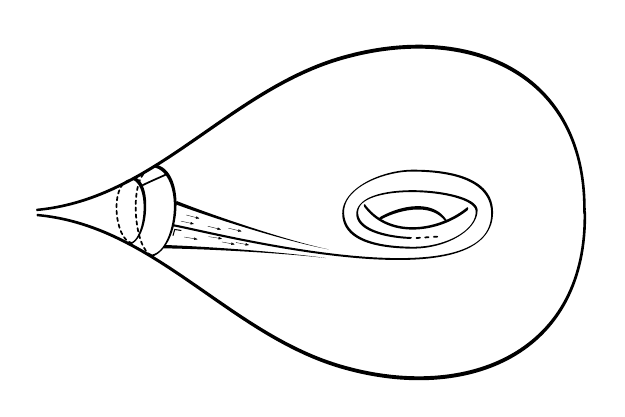}}}
\vspace{-30pt}
\end{center}
\caption{Geodesic rays in $\Lambda_{\rho,\epsilon}$}
\label{fig:measure2}
\end{figure}
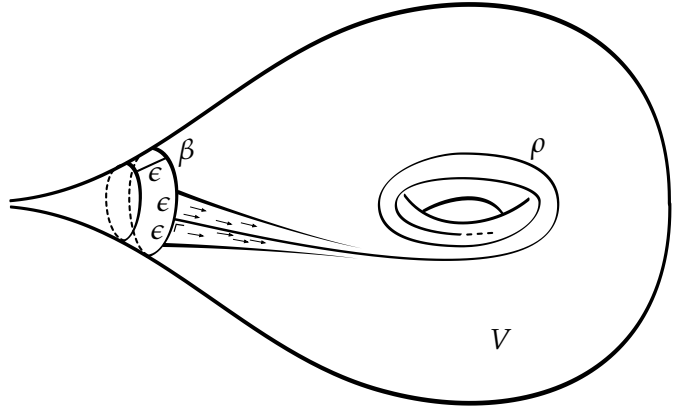
We have:
\begin{lemma} \label{lem: measure inequality} Let $\alpha$ be a boundary component of $V$ with $\lambda_{\alpha}>0$. There exists a constant $c>0$ that only depends on the geometry of $X$ so that for any 
$\delta >0$,
\begin{equation}
\text{Vol}(\Gamma_{\delta}) \geq \delta c \lambda_{\alpha}(\Lambda).
\end{equation}

\end{lemma}

\begin{proof} 
 We use the upper half-plane model and parametrize the set of oriented geodesics by $\{(x,y): x,y \in \widehat{\Bbb{R}}, x\neq y \}$, where $x$ and $y$ are the terminal and initial endpoints, respectively, of the geodesic. 
Consider the set of geodesics $\{(x,y): x \in A, y \in B \}$, where $A$ and $B$ are disjoint measurable subsets of $\Bbb{R}$ and consider an assignment of a geodesic segment on each of these geodesics of length 
 $\ell (x,y)$ for $x \in A$ and $y \in B$.
 In these coordinates, the Liouville measure of the vectors tangent to these geodesic segments is
\begin{equation}
\int_{A}\int_{B} \frac{\ell (x,y) dy dx}{(x-y)^2}.
\end{equation}

Since a boundary component of the concave core with grade $k$ is equidistant from the corresponding boundary component of grade $k^{\prime}$, and since orthorays emanating from these boundary components have a natural one to one correspondence, it is enough to prove the theorem for the boundary component being either a horocycle, geodesic, or the boundary of a disc centered at a cone point. 

We first consider the case that $\alpha$ is a horocycle of length one. 
We choose a normalized lift $\tilde{\alpha}$ to the upper half-plane so that $\tilde{\alpha}$ is the horocycle segment of height 1 between the points $i$ and $i+1$. All of the orthorays emanating from 
 $\tilde{\alpha}$ are vertical geodesics. Let $\rho$ be an orthoray in 
 $\Lambda$ with endpoint $x$. Then there exist positive constants $a_x, b_x$ so that the set of geodesic rays asymptotic to $\rho$ in $\Lambda_{\rho}$ have coordinates, 
 $\{(x, y): y \geq a_x\} \bigcup \{(x, y): y \leq -b_x\}$ (see Figure \ref{fig:measure2}). Using these coordinates the set of geodesics in $\Lambda_{\epsilon}$ asymptotic to $\rho$ is, $\{(x,y): y \geq a_x \ \text{or}\ y\leq b_x\}$. In fact, a straightforward calculation shows that 
 $a_{x}=x+(\frac{1+\epsilon^2}{\epsilon})$ and $b_{x}=x-(\frac{1+\epsilon^2}{\epsilon})$. Let $A$ be the set of endpoints of the orthorays in $\Lambda_{\epsilon}$. Since $0 \leq x \leq 1$, it follows that $\{(x, y):x \in A, y \geq a \} \subset \Lambda_{\epsilon}$ where $a=1+(\frac{1+\epsilon^2}{\epsilon})$. 
Noting that by construction $\ell (x,y) = \delta$ and that the unit tangent vectors of the geodesic segments are distinct (Lemma \ref{lem: distinct orthorays}), we have
\begin{equation}
\label{ }
 \text{Vol} (\Gamma_{\delta}) \geq 
\int_{A}\int_{a}^{\infty} \frac{\ell (x,y) dy dx}{(x-y)^2} =
\delta \int_{A}\int_{a}^{\infty} \frac{ dy dx}{(x-y)^2}
\end{equation}
\begin{equation}
=\delta \int_{A} \frac{dx}{a-x}
=\delta \int_{\alpha} \frac{\chi_A}{a-x} dx
\geq \delta \int_{\alpha} \frac{\chi_A}{a} dx
=\frac{\delta}{a} \lambda (A)=\Big( \frac{\delta \epsilon}{1+\epsilon +\epsilon^2} \Big) \lambda (A).
\end{equation}
where $A$ is the set of all endpoints of orthorays in the interval $[0,1]$, $\lambda (A)$ is the one dimensional measure on $A$, and $\chi_A$ is the characteristic function of $A$. Finally noting that $\lambda (A)=\lambda (\Lambda)$ and setting $c= \frac{ \epsilon}{1+\epsilon +\epsilon^2}$ finishes the horocycle boundary case. 

The cone point and boundary geodesic cases follow in much the same way as the horocyclic boundary case. Namely, for the case that $\alpha$ is the boundary of a disc centered at a cone point of angle $\theta$, we normalize a lift $\tilde{\alpha}$ so that it passes through $i$ and is symmetric about $i$. In the case of the boundary geodesic we normalize a lift so that it has endpoints $-1$ and $1$, and is symmetric about $i$ (see figure 3). In either case with these normalizations the geodesic projection from
$\tilde{\alpha}$ to $\mathbb{R}$ has bounded distortion and hence sets of measure zero correspond to sets of measure zero. Furthermore, the orthorays emanating from 
$\tilde{\alpha}$ are almost vertical. The rest of the argument is similar to the horocyclic case. We leave the details to the reader.\end{proof}

\begin{theorem} \label{thm: measure zero}
The measure of the set of orthorays that stay in the $\vec{k}$-concave core is zero.
\end{theorem}

\begin{proof}
Now let $\alpha$ be a boundary component of the $\vec{k}$-concave core $V$, and let $S(V)$ be the unit tangent bundle of $V$. Since $V$ is compact,
$\text{Vol}(S(V)) < \infty$. On the other hand, Lemma (\ref{lem: measure inequality}) tells us that, 
$
\infty > \text{Vol}(S(V))\geq \text{Vol} (\Gamma_{\delta}) 
\geq \delta c \lambda_{\alpha} (\Lambda),
$ for all $\delta$. Thus it must be that 
$\lambda_{\alpha} (\Lambda)=0.$ 
Finally, since the measure is zero for each boundary component it is zero for all orthorays emanating from the $\partial{V}$ that stay in $V$.
\end{proof}

\begin{remark}
	
\begin{enumerate}
	\item [(i)] If $\alpha$ is actually a cone point of angle $\pi/j$ or cusp (so the grade at $\alpha$ is $j$ or $\infty$), $\alpha$ has measure zero, but the set of directions from $\alpha$ has a natural positive measure (normalized to $1$ in the cusp case) and we can still identify this set of directions with the set of orthrays from $\alpha$. The above proof can be modified (Lemma 5.1 needs to be adjusted and replaced with some ergodicity result for the geodesic flow) to show that the subset $\Lambda$ of orthorays that stay in $V$ for infinite time has measure zero, provided that $V$ has at least one boundary component with positive length. On the other hand, if $X$ is in fact a model surface for the grading $\vec{k}$, then $V=X$ and in this case the set of orthorays which stay in $V$ for infinite time has full measure, as the number of orthorays with finite length is countable.
	\item [(ii)] Using the fact that a finitely generated Fuchsian group of the second kind must have a geodesic boundary component, say $\alpha$, and that the limit points correspond to either fixed points of parabolic elements or orthorays emanating from 
$\alpha$ which stay in the concave core of the quotient, Theorem \ref{thm: measure zero} yields the well-known fact that the measure of the limit set of a finitely generated Fuchsian group is either 0 or 1. 

\end{enumerate}
\end{remark}

We can now pass to the proof of the identity.

\subsection{Proof of the abstract identity}

The results of the previous section tell us that the complementary region to the gaps has measure $0$. To complete the proof of the theorem, we must show that the gaps are disjoint. 

Consider $\mu$, a $\vec{k}$-prime orthogeodesic between $\alpha$ and $\beta$, boundary elements of $V_{\vec{k}}(X)$. Let $Q_\mu \in \PP_{\vec{k}}(X)$ be the immersed pair of pants associated to $\mu$. 

To each $z\in \alpha$, we associate an orthoray $\rho_z$. The boundary $\alpha$ is divided in two parts: the gap $[x_\mu,y_\mu]$ associated to $\mu$, and the complementary region. 

{\it Claim 1:} An orthoray $\rho_z$ for $z\in [x_\mu,y_\mu]$ either hits $\alpha$ or $\beta$ in finite time and lies entirely inside $Q_\mu$, or is part of the measure $0$ set of infinite orthorays which stay in $V$ that remains inside $Q_\mu$. 

Note that this is true for any orthogeodesic, not just the $\vec{k}$-prime ones, but we're only interested in the gaps associated to $\vec{k}$-prime orthogeodesics. 

\begin{proof}[Proof of Claim 1.]
We argue in the pre-immersed pair of pants $\tilde{\tilde{P}}_\mu$. The boundary elements $\tilde{\tilde{\alpha}}$ and $\tilde{\tilde{\beta}}$ are the boundary of the $\vec{1}$-concave core of the underlying boundary elements. 

The gaps are the complementary region of the orthorays that begin by hitting $\tilde{\tilde{\gamma_\mu}}$ following a simple geodesic. Hence, if an orthoray emanating from the gap was to leave the pair of pants, it would have to first wrap around $\tilde{\tilde{\alpha}}$ or $\tilde{\tilde{\beta}}$ before leaving. But in that case they must intersect either $\tilde{\tilde{\alpha}}$ or $\tilde{\tilde{\beta}}$, as boundaries of the $\vec{1}$-natural collars. Otherwise the orthoray stays entirely inside the pair of pants and in particular belongs to the measure $0$ set. 
\end{proof}

Suppose the gaps for two $\vec{k}$-prime orthogeodesics, $\mu$ and $\eta$, intersect. 

{\it Claim 2:} The gaps for two $\vec{k}$-prime orthogeodesics, $\mu$ and $\eta$, are disjoint. 

\begin{proof}[Proof of Claim 2.]
If one of the gaps is strictly contained inside the other, then by Claim 1 above, the associated pairs of pants are strictly contained inside each other, and one of them is not $\vec{k}$-prime. 

Now suppose the gaps overlap. As the gaps are continuous intervals, this means that the endpoint of one gap, say that of $\mu$, is contained inside the gap of the other, so $\eta$. The endpoint of the gap of $\mu$, say $x_\mu$, corresponds to a ray that fellow travels $\gamma_\mu$ (in particular it wraps around it infinitely many times, see Figure \ref{fig:measure} for an illustration in the pre-immersed pair of pants). 

Now take any point $x$ in the gap of $\eta$ that does not belong to the gap of $\mu$. It hits $\gamma_\mu$ in finite time. Denote by $[x,x_\mu]$ the subset of $\alpha$ lying in the gap of $\eta$ between $x$ and $x_\mu$. The main observation is that the orthorays emanating from the points in this interval cover all points of $\gamma_\mu$. This follows from the continuity of the behavior of geodesics. The ray $\rho_x$ hits $\gamma_\mu$ in finite time, whereas $\rho_{x_\mu}$ wraps around $\gamma_\mu$ infinitely many times. As such, any point on $\gamma_\mu$ is hit by infinitely many rays emanating from the interval $[x,x_\mu]$. 

In particular, this implies that $\gamma_\mu$ is contained inside $Q_\eta$. From this we deduce that $Q_\mu$ is also contained inside $Q_\eta$. To see this, recall that $\gamma_\mu$ is the concatenation of $\alpha^{k_\alpha}, \mu, \beta^{k_\beta}$ and $\mu^{-1}$ and so $\gamma_\mu$ is contained if and only if $\alpha^{k_\alpha}, \mu, \beta^{k_\beta}$ are also contained. Now by construction, $Q_\mu$ is also contained inside $Q_\eta$, but this violates the $\vec{k}$-primality of $\mu$. 
\end{proof}
%%%%Computations of the measures%%%%%%%%%%%

\section{Computations of the measures}\label{sec:measures}

In this section we compute the measures and express the identities in terms of both traces and ortholengths.

\subsection{Computations}\label{subsec:Computations}

The immersed pair of pants point of view is essential in our computations. We compute the measure in the immersed pair of pants which has as boundary either cone points or boundary geodesics. We deduce the cusp case by a limit argument (but it can also be computed directly).

{\it Orthogeodesics between geodesic boundaries}

We begin with the case where $\mu$ is an oriented orthogeodesic between two boundary geodesics $\alpha$ and $\beta$ of indices $k_\alpha$ and $k_\beta$. Computations are done in the pre-immersed pair of pants of boundary curves $\tilde{\tilde{\alpha}}$, $\tilde{\tilde{\beta}}$ and $\tilde{\tilde{\gamma}}_\mu$. We denote by $\measure$ the measure (of the length) on the boundary geodesic and by $\bigmeasure$ the measure on the boundary of the concave core.

\begin{figure}[H]
%\ShowGrid
\leavevmode \SetLabels
\L(.29*0.58) $\frac{\tilde{\tilde{\alpha}}}{2}$\\
\L(.69*0.58) $\frac{\tilde{\tilde{\beta}}}{2}$\\
\L(.489*0.15) $\frac{\tilde{\tilde{\gamma}}}{2}$\\%
\L(.25*0.3) $h$\\%
\endSetLabels
\begin{center}
\AffixLabels{\centerline{\includegraphics[width=8cm]{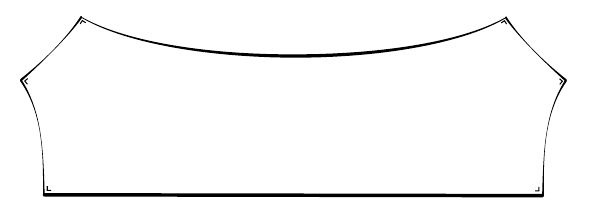}}}
\vspace{-30pt}
\end{center}
\caption{The hexagon in the pre-immersed pair of pants.}
\label{fig:hexagon}
\end{figure}
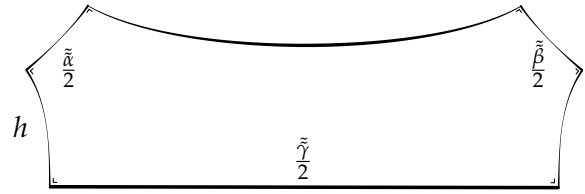
The key of the computation is to see that $\tilde{\tilde{\alpha}}$, $\tilde{\tilde{\beta}}$ and $\tilde{\tilde{\gamma}}_\mu$ form a pair of pants, hence their half-lengths satisfy the length relations of a right-angled hexagon, see Figure \ref{fig:hexagon}.

\begin{figure}[H]
%\ShowGrid
\leavevmode \SetLabels
\L(.35*0.405) $\frac{\measure(\mu)}{2}$\\%$\\%
\endSetLabels
\begin{center}
\AffixLabels{\centerline{\includegraphics[width=10cm]{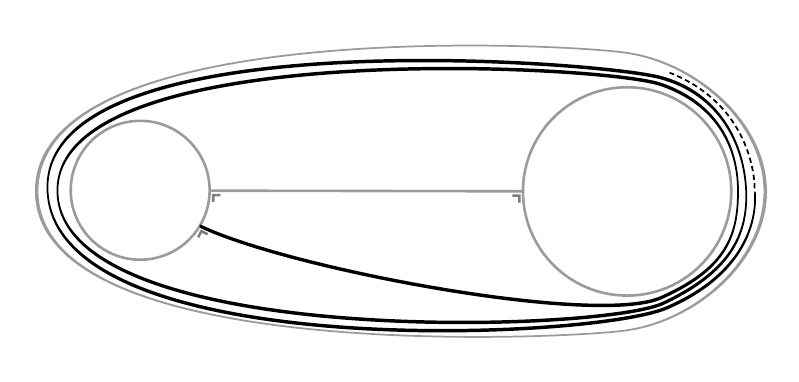}}}
%\vspace{-30pt}
\end{center}
\caption{The pre-immersed pair of pants where $\measure(\mu)/2$ appears. The full measure is given by the symmetric double.}
\label{fig:measure}
\end{figure}

Furthermore, $\measure(\mu)$ can be computed via an ideal right-angled quadrilateral immersed in the pair of pants, see Figures \ref{fig:measure} and \ref{fig:quad}. The sides labelled $h$ are useful to relate the different quantities.

\begin{figure}[H]
%\ShowGrid
\leavevmode \SetLabels
\L(.209*0.62) $\frac{\tilde{\tilde{\alpha}}-\measure(\mu)}{2}$\\%$\\%
\L(.195*0.31) $h$\\%
\endSetLabels
\begin{center}
\AffixLabels{\centerline{\includegraphics[width=10cm]{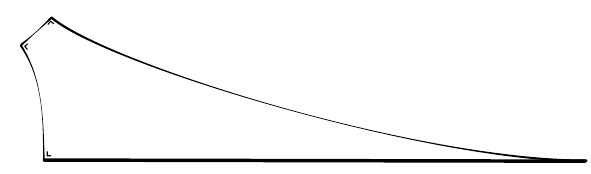}}}
\vspace{-30pt}
\end{center}
\caption{The immersed quadrilateral where $\measure(\mu)$ appears. }
\label{fig:quad}
\end{figure}

For the computation, we use the following half-trace notations:
$$
A:= \cosh\frac{\tilde{\tilde{\alpha}}}{2},\,B:= \cosh\frac{\tilde{\tilde{\beta}}}{2}, \,C:= \cosh\frac{\tilde{\tilde{\gamma_\mu}}}{2}
$$
and
$$
M:=
\cosh\frac{\measure(\mu)}{2}.
$$
Using standard trigonometry identities, we have
\begin{equation}\label{eq:start}
\sinh\left({\tilde{\tilde{\alpha}}}/2-\measure(\mu)/2\right) = \frac{\sqrt{A^2-1}\sqrt{C^2-1}}{\sqrt{A^2+B^2+C^2 + 2 ABC -1}}
\end{equation}
We set $p$ to be the polynomial
$$
p(x,y,z):=x^2+y^2+z^2+2xyz-1.
$$
Following further simplifications and manipulations we arrive at 
$$
M= \frac{A^2(C-\sqrt{C^2-1})+AB + \sqrt{C^2-1}}{\sqrt{p(A,B,C)}}.
$$
We can now express $M$ in terms of the lengths $\alpha$, 
$\beta$ and $\gamma_\mu$. Recall
$$
\tilde{\tilde{\alpha}}= k_\alpha \alpha, \tilde{\tilde{\beta}}= k_\alpha \beta, \mbox{ and } \tilde{\tilde{\gamma}}_\mu = \gamma_\mu,
$$
and we set $a,b,c$ to be 
$$
a:=\cosh(\alpha/2), b:= \cosh(\beta/2) \mbox{ and } c:=\cosh(\gamma_\mu/2).$$
Hence
$$
A=T_{k_\alpha}(a), B=T_{k_\beta}(b) \mbox{ and } C=c
$$
where $T_k(x)$ is the $k$-th Chebyshev polynomial of the first kind. We thus get the following expression for $M$:

$$
M= \frac{T_{k_\alpha}(a)^2(c-\sqrt{c^2-1})+T_{k_\alpha}(a)T_{k_\beta}(b) + \sqrt{c^2-1}}{\sqrt{p(T_{k_\alpha}(a),T_{k_\beta}(b),c)}}.
$$

Hence
\begin{equation}
\label{ }
\measure(\mu)= 2\arccosh\left(\frac{T_{k_\alpha}(a)^2(c-\sqrt{c^2-1})+T_{k_\alpha}(a)T_{k_\beta}(b) + \sqrt{c^2-1}}{\sqrt{p(T_{k_\alpha}(a),T_{k_\beta}(b),c)}}\right)
\end{equation}
and the measure on the boundary of the $\vec{k}$-concave core is
\begin{equation}
\label{ }
\bigmeasure(\mu)= \coth(k_\alpha \alpha/2)\measure(\mu).
\end{equation}
Expressed in terms of half-traces this becomes
\begin{eqnarray}
\label{ }
\bigmeasure(\mu) &=&
\frac{T_{k_\alpha}(a)}{\sqrt{T_{k_\alpha}(a)^2-1}}\measure(\mu)\\
&=& \frac{2 T_{k_\alpha}(a)}{\sqrt{T_{k_\alpha}(a)^2-1}}
\arccosh\left(\frac{T_{k_\alpha}(a)^2(c-\sqrt{c^2-1})+T_{k_\alpha}(a)T_{k_\beta}(b) + \sqrt{c^2-1}}{\sqrt{p(T_{k_\alpha}(a),T_{k_\beta}(b),c)}}\right).
\end{eqnarray}

We can also compute the above quantities using the orthogeodesic length on X. In order to do this we use the following hexagon formula that relates the sides,
$h, \frac{k_{\alpha} \alpha}{2}, \mu,$ and $\frac{k_{\beta} \beta}{2}$.

\begin{equation}\label{eq: }
\coth h \sinh \mu= A \cosh \mu - 
\frac{B} {\sqrt{B^2-1}} \sqrt{A^2-1}
\end{equation}

Using the quadrilateral formula along with the summation formula for $\cosh$ after some simplification we obtain
\begin{equation}\label{eq}
\coth h= A M - \sqrt{A^2-1} \sqrt{M^2-1}
\end{equation}

Eliminating $h$ in these two equations and rearranging
\begin{equation}\label{eq: gap measure}
\frac{A}{\sqrt{A^2-1}} M-\sqrt{M^2-1}=
\frac{1}{\sinh \mu} \left[\cosh \mu \frac{A}{\sqrt{A^2-1}}-
\frac{B}{\sqrt{B^2-1}} \right]
\end{equation}

For $k_\alpha <\infty$ (that is, $A <\infty$) we solve for $M$ in equation (\ref{eq: gap measure}) to get the two solutions

\begin{multline}\label{eq:gap measure for geodesic}
\frac{(A^2-1)}{\sinh \mu} \left[\frac{A^2 \cosh \mu}{A^2-1} -
\frac{AB}{\sqrt{A^2-1} \sqrt{B^2-1}}\right.\\ 
\left. \pm \sqrt{\left(\cosh \mu -\frac{AB}{\sqrt{A^2-1} \sqrt{B^2-1}}\right)^2-\frac{1}{(A^2-1)(B^2-1)}} \right]
\end{multline}

The geometrically relevant solution (henceforth denoted $M$) is the minus root.

\begin{eqnarray*}
M & = \cosh\left(\frac{\lambda(\mu)}{2}\right)\\
& = \frac{(A^2-1)}{\sinh \mu} & \left[\frac{A^2 \cosh \mu}{A^2-1} -
\frac{AB}{\sqrt{A^2-1} \sqrt{B^2-1}} \right. \\
& &- 
\left. \sqrt{\left(\cosh \mu -\frac{AB}{\sqrt{A^2-1} \sqrt{B^2-1}}\right)^2 
-\frac{1}{(A^2-1)(B^2-1)}} \right]
\end{eqnarray*}

{\it Letting $k_\alpha$ and $k_\beta$ go towards infinity}

Letting $k_{\beta}$ go to infinity corresponds to $B \rightarrow \infty$ in the expression above (Equation (\ref{eq:gap measure for geodesic}). We continue to denote the measure as $\lambda$, and as before $M= \cosh \frac{\lambda}{2}$. We see that 
$M=\coth \mu$ is the solution when $k_{\beta} \rightarrow \infty$. Put another way, 
$M=\coth \mu$ is a solution of the equation, 
\begin{equation}
\label{ }
\frac{A}{\sqrt{A^2-1}} M-\sqrt{M^2-1}=
\frac{1}{\sinh \mu} \left[\cosh \mu \frac{A}{\sqrt{A^2-1}}-
1 \right].
\end{equation} 
Note that in the case $\lambda(\mu)$ is the same measure as in the orthogeodesic identity of \cite{Basmajian}, namely
\begin{equation}
\lambda(\mu) = 2\log\left(\coth\frac{\mu}{4}\right).
\end{equation}
We next let $k_{\alpha} \rightarrow \infty$ (that is, $A \rightarrow \infty$) in Equation (\ref{eq: gap measure}) which reduces to a linear equation. Solving this equation we obtain
\begin{equation}
\label{eq:gap measure for k_alpha goes to infinity}
M=\cosh\frac{\lambda}{2}=
\frac{1}{2}\left[\frac{2\cosh^2 \mu -2\frac{B}{\sqrt{B^2-1}}
\cosh \mu +\frac{1}{B^2-1}}{\sinh \mu (\cosh \mu -1)}\right]
\end{equation}

We summarize these results in terms of the grades on $\alpha$ and $\beta$.
 
 \begin{proposition}
 Let $\mu$ be an orthogeodesic from a boundary geodesic 
 $\alpha$ with grade $k_{\alpha}$ to a boundary geodesic 
 $\beta$ with grade $k_\beta$.
 \begin{enumerate}
 \item For finite $k_\alpha$ and $k_\beta$, the gap measure is
\begin{multline} 
\cosh\left(\frac{\lambda(\mu)}{2}\right)=\\
\frac{\left(\sinh\frac{k_{\alpha}\alpha}{2}\right)^2}{\sinh \mu}
 \left[\coth^2 \left(\frac{k_{\alpha}\alpha}{2} \right)\cosh \mu 
- \coth \left( \frac{k_{\alpha}\alpha}{2}\right) \coth \left(\frac{k_{\beta}\beta}{2} \right) \right.\\
\left. -\sqrt{\left(\cosh \mu - \coth \left(\frac{k_{\alpha}\alpha}{2}\right) \coth\left(\frac{k_{\beta}\beta}{2}\right)\right)^2 -\frac{1}{\sinh^2\left( \frac{k_{\alpha}\alpha}{2} \right)
 \sinh^2 \left(\frac{k_{\beta}\beta}{2} \right)}}\right]
 \end{multline}
 
\item If $k_{\beta} \rightarrow \infty$ then 
$\cosh \left(\frac{\lambda(\mu)}{2}\right)=\coth \mu. $

\item If $k_{\alpha} \rightarrow \infty$ then 
\begin{equation}
\cosh\left(\frac{\lambda(\mu)}{2}\right)=
\frac{1}{2}\left[\frac{2 \cosh^2 \mu -2 \coth \frac{k_{\beta}\beta}{2} \cosh \mu + \frac{1}{\sinh^2 \frac{k_{\beta}\beta}{2} }}{\sinh \mu (\cosh \mu -1)}\right]
\end{equation}
\end{enumerate}
\end{proposition}

Since the prime orthogeodesics $\OO^{\prime} _{\vec{k}}(X),
\OO^{\prime} _{\vec{k}+1}(X),...$ form an exhaustion of 
$\OO(X)$, and since the respective gap measures converge to the gap measure for the orthospectrum identity:

\begin{theorem}
Let $X$ be a hyperbolic surface with geodesic boundary. 
The identity associated to the grade $\vec{k}$ converges to the orthospectrum identity, as 
$\vec{k} \rightarrow (\infty,...,\infty)$. 

\end{theorem}

We proceed in a similar manner in the other cases. We illustrate the appropriate geometric figures in each case, so that the detailed computations can be recuperated via standard trigonometry arguments. 

{\it Orthogeodesics between cone points}

Let $\mu$ be a $\vec{k}$-prime orthogeodesic between cone points $\alpha$ and $\beta$. Here we use the same notation as above, with the exception of 
$$
a:= \cos(\alpha/2) \mbox{ and } b:= \cos(\beta/2)
$$
where in order to not add more notation, we also use $\alpha$ and $\beta$ for the geometric measures of the angles. Here the mesure $\measure(\mu)$ is an angle and $\bigmeasure(\mu)$ is the length measure on the boundary of the $\vec{k}$-concave core. 

As above, we obtain:
\begin{equation}
\label{ }
\measure(\mu):= 2\arccos\left(\frac{T_{k_\alpha}(a)^2(c-\sqrt{c^2-1})+T_{k_\alpha}(a)T_{k_\beta}(b) + \sqrt{c^2-1}}{\sqrt{p(T_{k_\alpha}(a),T_{k_\beta}(b),c)}}\right)
\end{equation}
and
$$
\bigmeasure(\mu)= \cot (\alpha/2)\measure(\mu).
$$
hence
\begin{eqnarray}
\bigmeasure(\mu) & =&\frac{a}{\sqrt{1-a^2}}\measure(\mu)\\
& =& \frac{2a}{\sqrt{1- a^2}}\arccos\left(\frac{T_{k_\alpha}(a)^2(c-\sqrt{c^2-1})+T_{k_\alpha}(a)T_{k_\beta}(b) + \sqrt{c^2-1}}{\sqrt{p(T_{k_\alpha}(a),T_{k_\beta}(b),c)}}\right).\label{ }
\end{eqnarray}

{\it Orthogeodesics from a cone point to a geodesic}

Let $\mu$ be a $\vec{k}$-prime orthogeodesic between a cone point $\alpha$ and a closed geodesic $\beta$. Here we use the same notation as above, with the exception of 
$$
a:= \cos(\alpha/2) \mbox{ and } b:= \cosh(\beta/2).
$$
The quantity $\alpha$ is an angle and $\beta$ a length. Here the measure $\measure(\mu)$ is an angle and $\bigmeasure(\mu)$ is the length measure on the boundary of the $\vec{k}$-concave core.

As above, we obtain:
\begin{equation}
\label{ }
\measure(\mu)= 2\arccos\left(\frac{T_{k_\alpha}(a)^2(c-\sqrt{c^2-1})+T_{k_\alpha}(a)T_{k_\beta}(b) + \sqrt{c^2-1}}{\sqrt{p(T_{k_\alpha}(a),T_{k_\beta}(b),c)}}\right)
\end{equation}
and
$$
\bigmeasure(\mu)= \cot(\alpha/2)\measure(\mu).
$$
hence
\begin{eqnarray}
\bigmeasure(\mu) &= & \frac{a}{\sqrt{1-a^2}}\measure(\mu)\\
&= & \frac{2a}{\sqrt{1- a^2}}\arccos\left(\frac{T_{k_\alpha}(a)^2(c-\sqrt{c^2-1})+T_{k_\alpha}(a)T_{k_\beta}(b) + \sqrt{c^2-1}}{\sqrt{p(T_{k_\alpha}(a),T_{k_\beta}(b),c)}}\right).
\end{eqnarray}

{\it Orthogeodesics from a geodesic to a cone point}

Let $\mu$ be a $\vec{k}$-prime orthogeodesic between a geodesic $\alpha$ and a cone point $\beta$. Here we use the same notation as above, with the exception of 
$$
a:= \cosh(\alpha/2) \mbox{ and } b:= \cos(\beta/2)
$$
where $\alpha$ is a length and $\beta$ is an angle. Here the measure $\measure(\mu)$ is a length and $\bigmeasure(\mu)$ is the length measure on the boundary of the $\vec{k}$-concave core. 

As above, we obtain:
\begin{equation}
\label{ }
\measure(\mu)= 2\arccosh\left(\frac{T_{k_\alpha}(a)^2(c-\sqrt{c^2-1})+T_{k_\alpha}(a)T_{k_\beta}(b) + \sqrt{c^2-1}}{\sqrt{p(T_{k_\alpha}(a),T_{k_\beta}(b),c)}}\right)
\end{equation}
and

$$
\bigmeasure(\mu)= \coth(\alpha/2)\measure(\mu).
$$
hence
\begin{equation}\label{ }
\begin{split}
\bigmeasure(\mu) & = \frac{a}{\sqrt{1-a^2}}\measure(\mu)\\ 
& = \frac{2a}{\sqrt{a^2-1}}\arccosh\left(\frac{T_{k_\alpha}(a)^2(c-\sqrt{c^2-1})+T_{k_\alpha}(a)T_{k_\beta}(b) + \sqrt{c^2-1}}{\sqrt{p(T_{k_\alpha}(a),T_{k_\beta}(b),c)}}\right).
\end{split}
\end{equation}

Before completing the picture by passing to cusp limits, observe that the above measures are all expressions of the same quantity in $\C$ where we let $a$ and $b$ vary. To see this, assume that $-i\arccos(z)=\arccosh(z)$ (formally this depends on the position of $z$ in $\C$). Now since 
\begin{equation}
\label{ }
\frac{2a}{\sqrt{a^2-1}}= -i\frac{2a}{\sqrt{1-a^2}}
\end{equation}
the measures can be expressed in the same way. 

\subsection{Limiting to cusps}

We now compute the measures when one or both of the boundary elements $\alpha$ and $\beta$ are cusps. They are the measure of the geometric limit of the length measure on the boundary of the concave core, but also the analytic limit of the trace identities when $a$ or $b$ tends to 1 (and hence the cone point angle or the length goes to $0$). 

Observe that the limits are only problematic when $a$ tends to $1$. Indeed, the expressions are all continuous in $b=1$. 

We will take the limit when $\alpha$ and $\beta$ are closed geodesics. In the other cases, the computations are nearly identical. 

We begin with $k_\alpha,k_\beta$ finite. Recall that from Equation \ref{eq:start} above we have 
\begin{equation}
\begin{split}
\sinh\left({\tilde{\tilde{\alpha}}}/2-\measure(\mu)/2\right) = 
\frac{\sqrt{\left(\cosh\frac{k_{\alpha} \alpha}{2}\right)^2-1}\sqrt{\left(\cosh\frac{\gamma_\mu}{2}\right)^2-1}}
{\sqrt{ p(\cosh\frac{k_{\alpha} \alpha}{2}, \cosh\frac{k_{\beta} \beta}{2}, \cosh\frac{\gamma_\mu}{2})}}
\end{split}
\end{equation}
hence
$$
\bigmeasure(\mu) = \coth(k\alpha/2) \measure(\mu)
$$
where
\begin{equation}
\label{ }
\measure(\mu)=k_\alpha \alpha - 2\arcsinh\left( \frac{\sqrt{\left(\cosh\frac{k_{\alpha} \alpha}{2}\right)^2-1}\sqrt{\left(\cosh\frac{\gamma_\mu}{2}\right)^2-1}}{\sqrt{ p(\cosh\frac{k_{\alpha} \alpha}{2}, \cosh\frac{k_{\beta} \beta}{2}, \cosh\frac{\gamma_\mu}{2})}}
 \right).
\end{equation}
Taking the limit as $\alpha$ goes to $0$ this becomes
$$
\bigmeasure(\mu)= 2- 2\frac{\sqrt{\left(\cosh\frac{\gamma_\mu}{2}\right)^2-1}}{\sqrt{\left(\cosh\frac{k_{\beta} \beta}{2}\right)^2+\left(\cosh\frac{\gamma_\mu}{2}\right)^2+ 2 \cosh\frac{k_\beta \beta}{2} \cosh\frac{\gamma_\mu}{2}}}
$$
hence
\begin{equation}
\label{ }
\bigmeasure(\mu)= 2- 2\frac{\sinh\frac{\gamma_\mu}{2}}{\cosh\frac{k_{\beta} \beta}{2}+\cosh\frac{\gamma_\mu}{2}}=2 \frac{\cosh\frac{k_{\beta} \beta}{2}+e^{-\frac{\gamma_\mu}{2}}}{\cosh\frac{k_{\beta} \beta}{2}+\cosh\frac{\gamma_\mu}{2}}.
\end{equation}
In fine we obtain the following. 

{\it Orthogeodesics between a cusp and a geodesic}\\
The measure associated to $\mu$ is:
\begin{equation}
\label{ }
\bigmeasure(\mu)=2 \frac{\cosh\frac{k_{\beta} \beta}{2}+e^{-\frac{\gamma_\mu}{2}}}{\cosh\frac{k_{\beta} \beta}{2}+\cosh\frac{\gamma_\mu}{2}}.
\end{equation}

{\it Orthogeodesics between cusps}\\
When $\beta$ goes to $0$ from the above we obtain:
\begin{equation}
\label{ }
\bigmeasure(\mu)= \frac{4}{e^{\frac{\gamma_\mu}{2}}+1}.
\end{equation}

{\it Orthogeodesics between a cusp and a cone point}\\
The measure associated to $\mu$, computed the same way, is:
\begin{equation}
\label{ }
\bigmeasure(\mu)=2 \frac{\cos\frac{k_{\beta} \beta}{2}+e^{-\frac{\gamma_\mu}{2}}}{\cos\frac{k_{\beta} \beta}{2}+\cosh\frac{\gamma_\mu}{2}}.
\end{equation}

%[H: Note there are still cases missing: for instance what happens when $k_\alpha$ goes to infinity and $\alpha$ is a cusp? Maybe just leave these here for the dedicated reader to figure out with a brief explanation manuel? I drafted an explanation text below.]

{\it Cusps and infinite grades} 

In the presence of cusps, our identities allow for grades go to infinity, but the measures associated to orthogeodesics either go from or towards a cusp will always be zero if this is the case. Indeed, both the curves $\gamma_\mu$ and the orthogeodesic $\mu \subset X$ are of infinite length. Hence, the only relevant geometric input is the length of orthogeodesics on the concave core, but if at least one of the grades $k_\alpha$ or $k_\beta$ goes to infinity, then this length is also infinite.

%%%%%%%%The identities%%%%%%%%%%%%%%%%%

\section{The identities}\label{sec:identities}

We can now put all of this together to state the identities. Recall that $\OO'_{\vec{k}}(X)$ is the set of $\vec{k}$-prime orthogeodesics. 

\subsection{The general identity}

Let $X$ be an admissible hyperbolic surface, and $\vec{k}$ is a finite vector. Then
$$
\sum_{\mu \in \OO'_{\vec{k}}(X)} \bigmeasure(\mu) = \ell\left(\partial V_{\vec{k}}(X)\right)
$$
where
$$
\bigmeasure(\mu) = \frac{2a}{\sqrt{a^2-1}}\arccosh\left(\frac{T_{k_\alpha}(a)^2(c-\sqrt{c^2-1})+T_{k_\alpha}(a)T_{k_\beta}(b) + \sqrt{c^2-1}}{\sqrt{p(T_{k_\alpha}(a),T_{k_\beta}(b),c)}}\right)
$$
if $\mu$ leaves from a boundary geodesic $\alpha$,
$$
\bigmeasure(\mu) = \frac{2a}{\sqrt{1- a^2}}\arccos\left(\frac{T_{k_\alpha}(a)^2(c-\sqrt{c^2-1})+T_{k_\alpha}(a)T_{k_\beta}(b) + \sqrt{c^2-1}}{\sqrt{p(T_{k_\alpha}(a),T_{k_\beta}(b),c)}}\right)
$$
if $\mu$ leaves from a boundary geodesic $\alpha$, 
$$
\bigmeasure(\mu) =2 \frac{\cosh\frac{k_{\beta} \beta}{2}+e^{-\frac{\gamma_\mu}{2}}}{\cosh\frac{k_{\beta} \beta}{2}+\cosh\frac{\gamma_\mu}{2}}.
$$
if $\mu$ leaves from a boundary cusp $\alpha$ and goes to a boundary geodesic $\beta$, 
$$
\bigmeasure(\mu) =2 \frac{\cos\frac{k_{\beta} \beta}{2}+e^{-\frac{\gamma_\mu}{2}}}{\cos\frac{k_{\beta} \beta}{2}+\cosh\frac{\gamma_\mu}{2}}
$$
if $\mu$ leaves from a boundary cusp $\alpha$ and goes to a cone point $\beta$, and 
$$
\bigmeasure(\mu)= \frac{4}{e^{\frac{\gamma_\mu}{2}}+1}
$$
if $\mu$ goes between two cusps. 

\subsection{Cusp identities}

One interesting case of the identity is when the surface has only cusp boundary and $\vec{k}= \vec{1}$. In this case the identity counts the number of cusps :

\begin{theorem}\label{thm:cusps}
Let $X$ be a surface of genus $g$ and $n$ cusps with $\chi(X) = 2-2g -n \leq -1$ and $(g,n)\neq (0,3)$. Then 
$$
\sum_{\mu \in \OO'_{\vec{1}}(X)} \frac{2}{e^{\frac{\gamma_\mu}{2}}+1} = m.
$$
\end{theorem}

Note the above identity requires that the surface be of negative Euler characteristic and not be a pair of pants. 

In the case where the surface is homeomorphic to a pair of pants, and the boundary curves are all cusps, we get the following identity (the k-l-m identity):
\begin{theorem}\label{thm:klm}
Let $Y$ be the thrice punctured sphere and let $k\leq l \leq m$ be positive integers so that $ k + l + m >4$. Then 
$$
\sum_{\mu \in \OO'_{(k,l,m)}(Y)} \frac{2}{e^{\frac{\gamma_\mu}{2}}+1} = \frac{1}{k}+\frac{1}{l}+\frac{1}{m}.
$$
\end{theorem}

The particular case in Theorem \ref{thm:cusps} where $g=0$ and $m=4$, as well as the case in Theorem \ref{thm:klm} where $k=1, l=2$ and $m=2$, will both be discussed in the next section. 

\subsection{Euclidean identities}

When $(g,n)$ is equal to either $(0,3)$ or $(0,4)$, there are certain cases where the model surface is Euclidean. This happens in a finite number of cases, namely when the "half" angles sum to $\pi$ for $(0,3)$ and to $2\pi$ for $(0,4)$. Thus for gradings $(k_1,k_2,k_3)$ on three holed spheres satisfying
$$
\frac{1}{k_1}+\frac{1}{k_2}+\frac{1}{k_3}=2
$$
and for gradings $(k_1,k_2,k_3,k_4)$ satisfying
$$
\frac{1}{k_1}+\frac{1}{k_2}+\frac{1}{k_3}+\frac{1}{k_4}=4.
$$
The cases of equality are easy to identify. For $(g,n) = (0,3)$ we have
\begin{equation}
%(k_1,k_2,k_3) & =(1,1,\infty)\\
(k_1,k_2,k_3) =(1,2,2)
\end{equation}
and for $(g,n) = (0,4)$ this only occurs when
\begin{equation}
(k_1,k_2,k_3,k_4) =(1,1,1,1).
\end{equation}
%Of all of these cases, the case $(k_1,k_2,k_3) = (1,1,\infty)$ is special, as we shall briefly outline, and the others are all essentially the same.
We shall see how these two cases are related in the sequel.

{\it The case $(k_1,k_2,k_3)=(1,2,2)$.}

The model surface is a Euclidean orbifold surface (which is topologically a sphere) with cone points of angles $\pi$, $\pi/2$ and $\pi/2$ (see Figure \ref{fig:Model122}). Such a surface can be obtained by taking the symmetric double of a right-angled Euclidean isoceles triangle.

\begin{figure}[H]
%\ShowGrid
\leavevmode \SetLabels
\L(.34*0.95) $2$\\
\L(.34*0.03) $1$\\
\L(.65*0.03) $2$\\
\endSetLabels
\begin{center}
\AffixLabels{\centerline{\includegraphics[width=5cm]{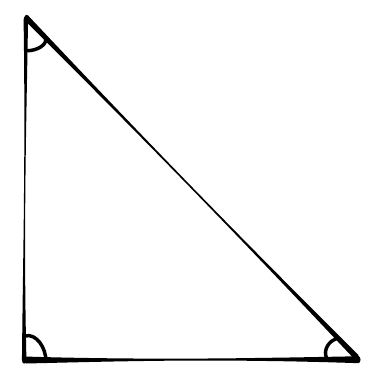}}}
\vspace{-30pt}
\end{center}
\caption{The model surface of a three holed sphere with grades $1, 2, 2$}
\label{fig:Model122}
\end{figure}

On this surface, there are infinitely many properly immersed geodesics between the cone points. One way to see this is to look at billiard paths on the underlying triangle. Each of these corresponds to an element of the index set. As will be explained in the next section, the identity corresponding to the boundary element with grade $1$ is in fact a version of the McShane identity. A particular case of the identity is when all boundary curves are cusps, in which case the identity is a case of Theorem \ref{thm:klm} above with $k=1,l=2$ and $m=2$, thus 
$$
\sum_{\mu \in \OO'_{(1,2,2)}(Y)} \frac{2}{e^{\frac{\gamma_\mu}{2}}+1} = 2.
$$

{\it The four holed sphere with grades $(1,1,1,1)$.}

In this case the identity is really a version of the McShane identity. Indeed, the model surface is a Euclidean spherical orbifold with four cone points of angle $\pi$ (see Figure \ref{fig:Model1111}). It can be obtained as the symmetric double of a Euclidean rectangle.

\begin{figure}[H]
%\ShowGrid
\leavevmode \SetLabels
\L(.34*0.95) $1$\\
\L(.34*0.03) $1$\\
\L(.65*0.03) $1$\\
\L(.65*0.95) $1$\\
\endSetLabels
\begin{center}
\AffixLabels{\centerline{\includegraphics[width=5cm]{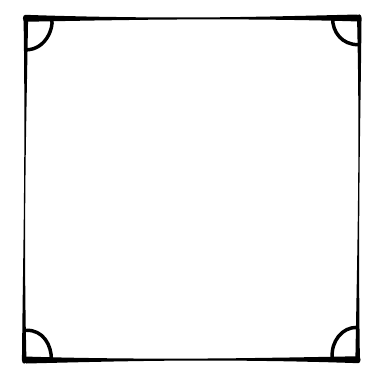}}}
\vspace{-30pt}
\end{center}
\caption{The model surface of a four holed sphere with grades $1, 1, 1, 1$.}
\label{fig:Model1111}
\end{figure}
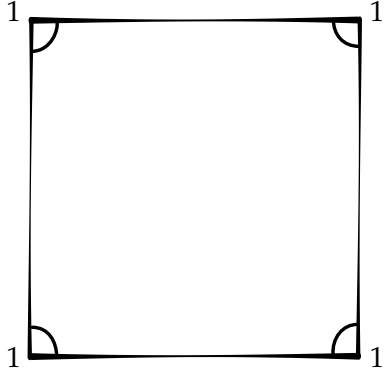
This surface is also a demi-translation surface, and geodesics between cone-points are saddle connections. In particular, all of these geodesics are simple and go between distinct cone points. In particular this implies that the orthogeodesics in the index set of this identity are exactly the set of simple orthogeodesics that lie between distinct boundary curves. 

For each such simple orthogeodesic $\mu$ corresponds a curve $\gamma_\mu$. As $\mu$ is simple and the initial and terminal grades are $1$, the curve $\gamma_\mu$ is simple too. If $X$ is a four holed sphere with four cusps as boundary, the identity becomes
$$
\sum_{\mu \in \OO'_{(1,1,1,1)}(X)} \frac{2}{e^{\frac{\gamma_\mu}{2}}+1} = 4.
$$
Note that each simple closed geodesic corresponds to exactly $4$ elements of the index set: the unique simple orthogeodesics disjoint from it. (There are $2$ unoriented such orthogeodesics, hence $4$ oriented simple orthogeodesics.) Rewriting the identity in terms of a sum of simple closed geodesics, this becomes
$$
\sum_{\gamma} \frac{1}{e^{\frac{\gamma}{2}}+1} = \frac{1}{2}.
$$
The relationship between this identity and the $(1,2,2)$ identity for three holed spheres is discussed in the next section. 
%%%%%%%Further properties of the identities%%%%%%%

\section{Further properties of the identities}\label{sec:finalobservations}

We end the paper with a few observations.

\subsection{Counting problems}

For both the Basmajian and McShane type identities, a lot is known about related counting problems. The counting problems are related to the index set: given $L$, how many elements of the index set are there of length less than $L$? The asymptotic behavior when $L$ goes to infinity is known for both cases. For the Basmajian identity this was solved by \cite{MartinEtAl} and the number grows exponentially (see \cite{He2,Parkkonen-Paulin} for related and more general results). For McShane type identities, there is asymptotic growth, which follows essentially from the work of Mirzakhani on the growth of simple closed geodesics \cite{MirzakhaniCount} (see \cite{Erlandsson-Souto} for a general reference on counting mapping class group orbits of curves). 

Of course, one has to be careful as how one is measuring the length of an element of the index set: if the index set is a collection of orthogeodesics between boundary curves, then the length of the orthogeodesic is a natural measure. If the index set is an embedded or immersed pair of pants, the total length of boundary curves would be a natural choice. In any event, we don't aim to prove precise asymptotic formulas for our identities, but simply state the problem and make some immediate observations. For exact formulas, growth will invariably be dependent on the grading. 

\begin{observation}
For fixed $\Sigma$, $\vec{k}$ and any admissible metric $X\in {\mathcal M}^{\vec{k}} (\Sigma)$ and the growth of the number of elements in the index set is roughly quadratic if the surface $M_{\vec{k}}$ is Euclidean, and roughly exponential when $M_{\vec{k}}$ is hyperbolic.\end{observation}

\begin{proof}
The length of an element in an index set can either be the length of an orthogeodesic, or the length of boundary curve of an immersed pair of pants. As we are only interested in rough growth, and the two are roughly comparable, either way of measuring length will do. 

Now observe than the number of index elements of length less than $L$ is roughly comparable to the corresponding number of properly embedded geodesics on $M_{\vec{k}}$. (More precisely, by roughly we mean, that for any $X$, there exist constants $c,C>0$ such that the number of elements of the index set is at least $cN(L)$ and at most $CN(L)$ where $N(L)$ the number of properly immersed geodesics on $M_{\vec{k}}$. 

To conclude it suffices to observe that $N(L)$ will grow either quadratically or exponentially depending on whether the (orbifold) universal cover is the Euclidean or hyperbolic plane. 
\end{proof}

\subsection{The under cover identity}

Let $\pi : \tilde{\Sigma} \to \Sigma$ be a topological regular covering map of order $r$ between punctured surfaces, possibly ramified in the punctures. We can realize this cover as an isometry between surfaces $\tilde{X}$ and $X$, possibly ramified over the boundary elements. So as not to introduce too much notation, we continue to denote this covering map as
$$\pi: \tilde{X} \to X.$$

We observe that a natural identification between orthogeodesics of $\tilde{X}$ and those of $X$. Indeed, an isometry must send an orthogeodesic of length $\ell$ on $\tilde{X}$ to an orthogeodesic of length $\ell$ on $X$, and conversely, an orthogeodesic on $X$ of length $\ell$ lifts to exactly $r$ copies of on $\tilde{X}$. The covering map $\pi$ also acts on boundary components, and doesn't necessarily act with the same degree of ramification. 

Putting together these observations shows that in fact there is a relationship between identities on $\tilde{X}$ or on $X$. Suppose we have an identity on $X$ with grades $\vec{k}$. This identity lifts to an identity on $\tilde{X}$ with grades $\vec{\tilde{k}}$ where the grades $\tilde{k}_i $ are computed as follows. Let $\pi(\tilde{\beta}_i)= \beta_j$ be the image of boundary component $\tilde{\beta}_i$ by $\pi$ and let $r_i\in \N$ be the degree of ramification. Then $\tilde{k}_i=r_i k_j$. Equivalently, this corresponds to isometric regular covers of the model surfaces.

An example is given by two of the Euclidean identities. Indeed it is easy to see that the model surface of the three holed sphere with grades $1,2,2$ admits the model surface of the four holed sphere with grades $1,1,1,1$ as a two fold cover. But for the latter, you need to choose it to be the symmetric double of a square, and not just the symmetric double of a rectangle. The $(1,2,2)$ identity on the three holed sphere then lifts to the $(1,1,1,1)$ identity on the four holed sphere, but only to surfaces with the corresponding self isometry. 

\subsection{Final remarks}

Many of these ideas can be used to obtain a generalization of Bridgeman's identity, in particular to surfaces with cusps. The decomposition, in this case, is a decomposition of the unit tangent bundle with identical index sets to the ones explained here. The proof and explicit computations of the measures will be the object of forthcoming paper.

{\em Addresses:}\\
A. B.: The Graduate Center, CUNY, 365 Fifth Ave., N.Y., N.Y., 10016 {\it and}\\
Hunter College, CUNY, 695 Park Ave., N.Y., N.Y., 10065, USA\\
%Graduate Center and Hunter College, CUNY, New York, USA \\
H.P.: Department of Mathematics, University of Luxembourg, Luxembourg \\
S.T.: National University of Singapore, Singapore\\
{\em Emails:}
 \href{mailto:abasmajian@gc.cuny.edu}{abasmajian@gc.cuny.edu},
 \href{mailto:hugo.parlier@unifr.ch}{hugo.parlier@uni.lu}, \href{mailto:mattansp@nus.edu.sg}{mattansp@nus.edu.sg}\\

\end{document}